\documentclass[12pt,bookmark=true]{article}

\RequirePackage{geometry}
\usepackage{graphicx}				

\usepackage{amssymb}
\usepackage{amsmath}
\usepackage{color}
\usepackage{amsthm}
\usepackage{enumerate}
\usepackage{lineno}
\usepackage{indentfirst}
\usepackage{mathrsfs}
\usepackage{amsfonts}
\usepackage{amssymb}
\usepackage{amsthm}
\usepackage{geometry}
\usepackage{hyperref}

\newtheorem{thm}{\noindent Theorem}[section]
\newtheorem{lem}[thm]{\noindent Lemma}

\newtheorem{pro}[thm]{\noindent Proposition}
\newtheorem{de}[thm]{\noindent Definition}
\theoremstyle{definition}{
\newtheorem{remark}[thm]{\noindent Remark}
}
\numberwithin{equation}{section}
\def\R{\mathbb{R}}
\def\O{\mathcal{O}}
\hypersetup{hidelinks}
\begin{document}
	\title{Hierarchical exact controllability for a parabolic equation with Hardy potential}
	\author{
		Haiyang Lin\footnote{Email address: linhaiyang2002@163.com}, Bo You\footnote{Corresponding author; Email address: youb2013@xjtu.edu.cn}
		\\
		{\small School of Mathematics and Statistics, Xi'an Jiaotong University} \\
		\small Xi'an, 710049, P. R. China}
	
	\maketitle
	
\begin{center}
\begin{abstract}
The main objective of this paper is to study the hierarchical exact controllability for a parabolic equation with Hardy potential by Stackelberg-Nash strategy. In linear case, we employ Lax-Milgram theorem to prove the existence of an associated Nash equilibrium pair corresponding to a bi-objective optimal control problem for each leader, which is responsible for an exact controllability property. Then the observability inequality of a coupled parabolic system is established by using global Carleman inequalities, which results in the existence of a leader that drives the controlled system exactly to any prescribed trajectory. In semilinear case, we first prove the well-posedness of the coupled parabolic system to obtain the existence of Nash quasi-equilibrium pair and show that Nash quasi-equilibrium is equivalent to Nash equilibrium. Based on these results, we establish the existence of a leader that drives the controlled system exactly to a prescribed (but arbitrary) trajectory by Leray-Schauder fixed point theorem.
			

		\end{abstract}
	\end{center}
	
	\section{Introduction}
	\noindent  
	In the classical control theory, the problems we usually consider only involve a unique target with a single control, and the predetermined target is to minimize a cost functional in a prescribed family of admissible controls. However, the problems with several different and even contradictory control objectives may be more common in reality. For example, we intend to keep reasonable humidity in some areas of the room during the whole time interval $(0,T),$ and drive the humidity in a room to a desired target at the time $T$ by  humidifier and dehumidifier acting on several small subdomains. To deal with such multi-objective problems, we will make use of Stackelberg-Nash strategy which combines the Stackelberg hierarchical-cooperative strategy \cite{von1934marktform} and non-cooperative optimization techniques proposed by Nash \cite{nash1951non}. The general idea of this strategy is that the leader (the main control) makes the first movement and then the followers (the secondary controls) react optimally to the action of the leader. 
	
	As a precedent, J.  L.  Lions \cite{lions1994hierarchic} has done pioneering works in hierarchic control of PDEs, where he employed Stackelberg strategy and considered two controls (one leader and one follower) in the context of wave PDEs. In \cite{diaz2002neumann,diaz2004approximate}, J.  I.  Di\'{a}z et.al have established the approximate controllability of the systems by using Stackelberg-Nash strategy. The Stackelberg-Nash exact controllability to the trajectories of linear and semilinear parabolic equations have been given in \cite{araruna2020hierarchical,araruna2017new,araruna2015stackelberg,DJOMEGNE2025128799}; the problems with distributed controls was analyzed in \cite{araruna2017new,araruna2015stackelberg}, while \mbox{F.  D.  Araruna \textit{et al.}} \cite{araruna2020hierarchical} have dealt with the problems with both distributed and boundary controls and \mbox{L. Djomegne \textit{et al.}} \cite{DJOMEGNE2025128799} have considered the problem with all controls acting on the boundary. Moreover, \mbox{F.  D.  Araruna \textit{et al.}} have also considered hierarchic control for semilinear wave equations in \cite{araruna2018hierarchic}. For more works on hierarchical controllability of wave equations, we refer the readers to \cite{DEJESUS2016377} and \cite{DETERESA2025103428}. We also refer the readers to \cite{guillen2013approximate,hernandez2016hierarchic} to study the hierarchical controllability of two coupled equations of Stokes systems and parabolic systems.  In addition, the hierarchical controllability of the fourth order parabolic equations was analyzed by F. Li and B. You in \cite{li2025hierarchical}. Also we would like to mention the work \cite{ramos2002nash,ramos2002pointwise} by \mbox{Ramos \textit{et al.}} where they study Nash equilibrium for constraints given by linear parabolic and Burger's equations from the points of theoretical and numerical view.
	
	In this paper, we mainly consider the hierarchic exact controllability of the parabolic equations with singular potentials. More precisely, we focus on the Hardy potential $\frac{\mu}{|x|^2}$ which usually appears in the linearization of standard combustion models (see \cite{bebernes2013mathematical,brezis1997blow,dold1998rate,galaktionov1997continuation,peral1995stability}) and the context of quantum mechanics (see \cite{baras1984remarks,de2004bound}). In 1984, Baras and Goldstein \cite{baras1984heat} studied the heat equations with such inverse-square singular potentials and proved the existence as well as the non-existence of positive solutions depend on the value of the parameter $\mu.$ Later,  E.  Zuazua et.al \cite{vazquez2000hardy} complemented the well-posedness results of \cite{baras1984heat,baras1984remarks} and precisely described the functional spaces in which such problems are well-posed, especially for the critical case. Furthermore,  E.  Zuazua et.al \cite{vancostenoble2008null} have consider the controllability properties of such equations by obtainning Carleman inequalities for one-dimensional problems with singular potentials. Moreover, S. Ervedoza \cite{ervedoza2008control} extended the Carleman estimates for one-dimensional case to the N-dimensional case and deduced a null controllability result for the same problem with control supported in any nonempty subdomain. We also refer the readers to \cite{vancostenoble2011lipschitz} in which the author has considered the inverse source problems for such parabolic equations. 
	
	Let $N\ge 3$ be given and $\Omega\subset\R^N$ is a bounded domain with boundary $\Gamma$ of class $C^2$ such that $0\in \Omega.$ Given $T>0,$ we will set $Q:=\Omega \times (0,T)$ and $\Sigma:=\Gamma\times (0,T).$ Assume that $\O, \O_1, \O_2\subset\Omega$ are three small open nonempty sets. In this paper, we will analyze the hierarchic exact controllability of the following system:
	\begin{align}\label{eq1}
		\begin{cases}
			y_t-\Delta y-\frac{\mu}{|x|^2}y=F(y)+f1_{\O}+v^11_{\O_1}+v^21_{\O_2},&(x,t)\in Q,\\
			y(x,t)=0,&(x,t) \in\Sigma,\\
			y(x,0)=y_0(x), &x\in\Omega,
		\end{cases}
	\end{align}
	where $f\in L^2(\O\times(0,T))$ is the leader and $v^1 \in L^2(\O_1\times(0,T)),v^2\in L^2(\O_2\times(0,T))$ are the followers, $y_0\in L^2(\Omega)$ is given, $F\in C^1(\R)\cap W^{1,\infty}(\R),$ $0\le\mu\le\mu^*(N)$, here $\mu^*(N)=\frac{(N-2)^2}{4}$ is the optimal constant in the following Hardy inequality (see for example \cite{evans2022partial}):
	\[\mu^*(N)\int_{\Omega}\frac{z^2}{|x|^2}\,dx\le\int_{\Omega}|\nabla z|^2\,dx,\quad \forall z\in H^1_0(\Omega).\]
	The notation $1_A$ indicates the characteristic function of the set $A$.
	
	
	Let $\O_{1,d}$ and $\O_{2,d}$ be nonempty open subsets with $\O\cap \O_{i,d}\neq \emptyset$ and $0\notin \overline{\O}_{i,d}$. We define the secondary cost functionals by
	\begin{align}
		J_i(f;v^1,v^2)=\frac{1}{2}\iint_{\O_{i,d}\times (0,T)}|y-y_{i,d}|^2\,dxdt+\frac{\alpha_i}{2}\iint_{\O_{i}\times (0,T)}|v^i|^2\,dxdt,\quad i=1,2,\label{sc1}
	\end{align}
	where $y_{i,d}\in L^2(\O_{i,d}\times (0,T))$ are given functions and $\alpha_i$ are positive constants. Let us also introduce the main cost functional:
	\begin{align*}
		J(f)=\frac{1}{2}\iint_{\O\times (0,T)}|f|^2\,dxdt.
	\end{align*}
	
	The control process can be divided into two steps. First of all, for each choice of the leader $f$, we try to find a Nash equilibrium pair $(\bar{v}^1(f),\bar{v}^2(f))$ for the cost functionals $J_i\ (i=1,2)$. That is, for any fixed $f\in L^2(\O\times(0,T))$, we would like to prove that there exist $\bar{v}^1\in L^2(\O_{1}\times (0,T))$ and $\bar{v}^2 \in L^2(\O_{2}\times (0,T))$, depending
	on $f$, satisfying simultaneously
	\begin{equation}\label{Nash1}
		\begin{aligned}
			&J_1(f;\bar{v}^1,\bar{v}^2)\le J_1(f;v^1,\bar{v}^2),\quad \forall v^1\in L^2(\O_{1}\times (0,T)),\\ &J_2(f;\bar{v}^1,\bar{v}^2)\le J_1(f;\bar{v}^1,v^2),\quad \forall v^2\in L^2(\O_{2}\times (0,T)).
		\end{aligned}
	\end{equation}
	Second,  let $\bar{y}$ be the solution of the following problem
	\begin{align}
		\begin{cases}
			\bar{y}_t-\Delta \bar{y} -\frac{\mu}{|x|^2}\bar{y}=F(\bar y) ,&(x,t)\in  Q,\\
			\bar{y}(x,t)=0,&(x,t) \in \Sigma,\\
			\bar{y}(x,0)=\bar{y}_0(x),& x\in \Omega,
		\end{cases}\label{tr}
	\end{align}
	where $\bar{y}_0\in L^2(\Omega)$. After proving that there exists at least one Nash equilibrium for each $f$, we would like to look for a leader $\hat{f}$ such that 
	\begin{align}\label{op}
		J(\hat{f})=\min_{f}J(f)
	\end{align}
	subject to the exact controllability condition
	\begin{align}\label{exa}
		y(x,T)=\bar y(x,T)\quad \text{for \textit{a.e.}}\  x\in \Omega.
	\end{align}
The main difficulties and novelties of this paper are summarized as follows.
\begin{enumerate}[\rm{(}1\rm{)}]
\item Since there are three controls appeared in problem \eqref{eq1}, such that the investigation of the hierarchical exact controllability of problem \eqref{eq1} is reduced to the study of the null controllability of a coupled system. Thus,  we need to establish an observability inequality for a coupled system of parabolic equations with Hardy potentials, which is more involved than the null controllability of a single parabolic equation.
\item Compared with other literature about the hierarchical null controllability of PDEs, we use a simpler method to establish the equivalence between the Nash equilibrium pair and the Nash quasi-equilibrium pair under weaker conditions than those in existing studies.
\item If the nonlinearity function $F(y)$ is replaced by $F(y,\nabla y),$ due to the presence of the Hardy potential term $\frac{u}{|x|^2},$ we can not obtain the $L^2_tH^s_x$-regularity of weak solution for problem \eqref{eq1} with intial data $y_0\in H^s(\Omega)$ for some $s>1,$  such that the well-posedness of optimality system can not obtained by the Leray-Schauder's fixed points Theorem for the semilinear case. Thus, we only consider the case that the nonlinearity is $F(y).$

\item If the term $a(x,t)\nabla y$ is added to problem \eqref{eq1} with $a\in L^{\infty}(Q)$, we have to establish the global Carleman estimates for the heat equation with Hardy potentials and $H^{-1}(\Omega)$ external force to obtain the observability inequality for the adjoint system. However,  due the degeneracy of the weight function at the origin, the exponents of $s$ in front of the terms
\begin{align*}
\int_Qe^{-2\sigma}|u(x,t)|^2\,dxdt\,\,\,\textit{or}\int_Qe^{-2\sigma}|\nabla u(x,t)|^2\,dxdt
\end{align*}
in global Carleman estimate for the heat equation with Hardy potential obtained by interpolation inequality is not enough to deduce some weighted energy estimates satisfied by solutions of the corresponding dual system based on the established global Carleman inequality and the theory of optimal control,  such that we can not obtain the global Carleman inequality for the heat equation with Hardy potential and $H^{-1}(\Omega)$ external force.  Thus,  we can not deal with the term $\nabla\cdot (a(x,t)u(x,t))$ in the dual system to obtain the desired observability inequality. Base on the above reasons, we only consider problem \eqref{eq1} without gradient term. 
\end{enumerate}
The rest of this paper is organized as follows. In Section 2, we recall the well-posedness results and global Carleman inequalities for the parabolic equation with Hardy potential. Section 3 is devoted to prove the existence and uniqueness of Nash equilibrium by Lax-Milgram theorem for any given leader $f,$ and prove the exact controllability to trajectory of problem \eqref{eq1} in the linear case. In Section 4, we analyze the relation between Nash equilibrium and Nash quasi-equilibrium, and establish the exact controllability to trajectory of problem \eqref{eq1} by using Leray-Schauder fixed-point argument in the semilinear case.

	Throughout this paper, the following notations will be used:
	\[\|u\|=\|u\|_{L^2(\Omega)},\quad (u,v)=(u,v)_{L^2(\Omega)}.\]
	Moreover, we use $C$ to denote a general positive constant that will in general stand for different constants in different lines.	
	\section{Preliminaries}\label{preliminary}
	In this section,  we will recall some lemmas used in the sequel.  First of all, we give the well-posedness result of the following problem 
	\begin{equation}\label{ba}
		\begin{cases}
			u_t-\Delta u-\frac{\mu}{|x|^2}u=g,& (x,t)\in Q,\\
			u(x,t)=0,& (x,t)\in \Sigma,\\
			u(x,0)=u_0(x),&x\in \Omega.
		\end{cases}
	\end{equation}
	\begin{lem}[see \cite{vazquez2000hardy}]\label{well}
		 Assume that $g \in L^{2}(Q).$ 
		\begin{enumerate}[(i)]
			\item If $\mu < \mu^{*}(N)$. Then for any $u_0 \in L^{2}(\Omega),$ there exists a unique weak solution of problem \eqref{ba},  such that
			\[
			u \in L^{2}(0,T;H_{0}^{1}(\Omega))\cap H^1(0,T;H^{-1}(\Omega)).
			\]
			\item For $\mu = \mu^{*}(N),$ let $\mathcal{M}$ be the Hilbert space obtained by the completion of $H_{0}^{1}(\Omega)$ with respect to the following norm
			\[
			\|u\|_{\mathcal{M}} = \left(\int_{\Omega} |\nabla u|^{2}-\mu^{*}(N)\frac{|u|^{2}}{|x|^{2}}\, dx\right)^{\frac{1}{2}}.
			\]
			Then for any $u_0 \in L^{2}(\Omega),$ there exists a unique weak solution of problem \eqref{ba} such that
			\[
			u \in L^{2}(0,T;\mathcal{M})\cap H^1(0,T;\mathcal{M}^{\prime}),
			\]
			where $\mathcal{M}^{\prime}$ is the dual space of $\mathcal{M}.$ Furthermore, we have the following compact embedding
			\begin{equation}
				\mathcal{M}\hookrightarrow\hookrightarrow L^2(\Omega)\hookrightarrow\hookrightarrow\mathcal{M}^{\prime}.
			\end{equation}
		\end{enumerate}
	\end{lem}
Without loss of generality, we always assume that 
\begin{align*}
\overline{B(0,1)}\subset\Omega,\,\,\,\, \overline{B(0,1)}\cap \O_{i,d}=\emptyset
\end{align*}
for any $i=1,2$ and denote by $\widetilde{\Omega}=\Omega\backslash \overline{B(0,1)}.$ In order to state the global Carleman estimate, we need to introduce a special weight function.
	\begin{lem}[see \cite{ervedoza2008control,fursikov1996controllability}]\label{lem1}
		Let $\omega_0\subset\subset \Omega$ be an arbitrary fixed nonempty open subset. Then there exists a smooth function $\Psi$ satisfying
		\begin{equation*}
			\begin{cases}                      					
\Psi(x)=\ln(|x|),& \quad x\in B(0,1),\\  
 \Psi(x)=0, &\quad x\in \Gamma,\\           
  \Psi(x)> 0,&\quad x\in \widetilde{\Omega},\\                                            
  |\nabla \Psi|>0,&\quad x\in \overline{\Omega} \backslash \omega_0.                         \end{cases}
		\end{equation*}	
	\end{lem}
Moreover,  we also introduce the following weight functions:
	\begin{align*}
		\theta(t)=t^{-3}(T-t)^{-3},\quad \sigma(x,t)=s\theta(t)(e^{2\lambda\sup\Psi}-\frac{1}{2}|x|^2-e^{\lambda\Psi}),\quad \Phi(x)=e^{\lambda \Psi(x)},
	\end{align*}
	where $s$ and $\lambda$ are positive parameters.

In what follows, we will recall the global Carleman inequality for the heat equations with Hardy potential. 	
	\begin{lem}[see \cite{ervedoza2008control,vancostenoble2011lipschitz}]\label{carle}
Assume that $u_0\in L^2(\Omega),$ $g\in L^2(Q)$ and let $\sigma,$ $\theta,$ $\Psi,$ $\Phi$ be defined as above.  If $u$ is the unique weak solution of problem
		\begin{equation}\label{careq}
			\begin{cases}
				-u_t-\Delta u-\frac{\mu}{|x|^2}u=g,& (x,t)\in Q,\\
				u(x,t)=0,& (x,t)\in \Sigma,\\
				u(x,T)=u_0,&x\in \Omega,
			\end{cases}
		\end{equation}
then there exist two positive constants $C>0$ and $\lambda_0>1$ satisfying for any $\lambda\geq \lambda_0,$ we can choose a positive constant $s_0(\lambda)>0,$ such that for any $s\geq s_0(\lambda),$
		\begin{equation}\label{carleman}
			\begin{aligned}
			&s^3\iint_{Q}\theta^3e^{-2\sigma}|x|^2u^2\,dxdt+s^3\lambda^4\iint_{\widetilde{\Omega}\times(0,T)}\theta^3\Phi^3e^{-2\sigma}u^2\,dxdt+s\iint_{Q}\theta e^{-2\sigma}\frac{u^2}{|x|}\,dxdt\\
			&+s(\mu^*(N)-\mu)\iint_{Q}\theta e^{-2\sigma}\frac{u^2}{|x|^2}\,dxdt+s\lambda^2\iint_{\widetilde{\Omega}\times(0,T)}\theta\Phi e^{-2\sigma}|\nabla u|^2\,dxdt\\
			\le&C\iint_{Q}e^{-2\sigma}g^2\,dxdt+Cs^3\lambda^4\iint_{\omega_0\times(0,T)}\theta^3\Phi^3e^{-2\sigma}u^2\,dxdt.
			\end{aligned}
		\end{equation}
	\end{lem}

	Moreover, we introduce the Cacciopoli's inequality which is also relevant to the work in this paper: 
	\begin{lem}[see \cite{ervedoza2008control,vancostenoble2008null}]\label{Cap}
		Suppose that $\omega_0$ and $\omega$ are arbitrary open sets satisfying $\omega_0\subset\subset\omega\subset\subset\Omega$ and $0\notin \widetilde{\omega}$. Let $\widetilde{\sigma}:\Omega\times (0,T)\to \R^{+}$ be a function such that
		\begin{equation}
			\widetilde{\sigma}(x,t)\to+\infty\quad \text{as}\ t\to 0^{+}\ \text{and }\ t\to T^{-}.
		\end{equation}
		There exists a positive constant $C$ independent of $\mu\le \mu^*(N)$ such that any solution $u$ of problem \eqref{careq} satisfies the following inequality
		\begin{equation}
			\iint_{\omega_0\times(0,T)}e^{-2\widetilde{\sigma}}|\nabla u|^2\,dxdt\le C\iint_{\omega\times (0,T)}e^{-\widetilde{\sigma}}|u|^2\,dxdt+C\iint_{\omega\times(0,T)}e^{-2\widetilde{\sigma}}g^2.
		\end{equation}
	\end{lem}

	\section{The linear case} 
	In this section,  we will study the hierarchical exact controllability of problem \eqref{eq1} in the linear case ($F\equiv 0$).  To begin with,  we introduce a new variable $z:=y-\bar y,$ then $z$ satisfies the following problem
	\begin{align}\label{eqn}
		\begin{cases}
			z_t-\Delta z-\frac{\mu}{|x|^2}z=f1_{\O}+v^11_{\O_1}+v^21_{\O_2} ,&(x,t)\in Q,\\
			z(x,t)=0,& (x,t)\in\Sigma,\\
			z(x,0)=z_0(x),&x\in \Omega,
		\end{cases}
	\end{align}
	where $z_0=y_0-\bar{y}_0$. Then the exact controllability to the trajectory of system (\ref{eq1}) is equivalent to the null controllability of problem \eqref{eqn}. More precisely, the condition \eqref{exa} is equivalent to
	\begin{align}\label{nullctrl}
		z(x,T)=0 \quad  \text{for \textit{a.e.}}\  x\in\Omega.
	\end{align}
Denote by $z_{i,d}=y_{i,d}-\bar{y},$ then we can reformula the cost functionals $J_i$ as follows:
	\begin{align}\label{sc}
		J_i(f;v^1,v^2)=\frac{1}{2}\iint_{\O_{i,d}\times (0,T)}|z-z_{i,d}|^2\,dxdt+\frac{\alpha_i}{2}\iint_{\O_{i}\times (0,T)}|v^i|^2\,dxdt,\quad i=1,2. 
	\end{align}

	
	\subsection{Existence and uniqueness of the Nash equilibrium}
	In this subsection, we will prove the existence and uniqueness of Nash equilibria. Let $f\in L^2(\O\times (0,T))$ be fixed and define 
	\begin{align}\label{H}
		H_i=L^2(\O_{i}\times (0,T)),\,\,\, H=H_1\times H_2.
	\end{align} 
	Note that, in this linear case, the cost functionals $J_i$ are convex and continuously differentiable. Thus, \eqref{Nash1} is equivalent to
	\begin{align}\label{queq}
		D_iJ_i(f;\bar{v}^1,\bar{v}^2)\cdot v^i=0,\quad \forall v^i\in L^2(\O_{i}\times (0,T)).
	\end{align}
	Therefore, $(\bar{v}^1,\bar{v}^2)$ is a Nash equilibrium if and only if for any $v^i\in H_i$
	\begin{align}\label{Nash2}
		\iint_{\O_{i,d}\times (0,T)}(z(f;\bar{v}^1,\bar{v}^2)-z_{i,d})w^i\,dxdt+\alpha_i\iint_{\O_i\times(0,T)}\bar{v}^iv^i\,dxdt=0,
	\end{align}
	where $w^i$ solves the following system
	\begin{align}
		\begin{cases}
			w^i_t-\Delta w^i-\frac{\mu}{|x|^2}w^i=v^i1_{\O_i},&(x,t)\in Q,\\
			w^i=0,&(x,t) \in \Sigma,\\
			w^i(x,0)=0,&x\in\Omega.
		\end{cases}\label{eq2}
	\end{align}
	Employing energy estimate, we can define the operators $L_i\in L(H_i;L^2(Q))$ by
	\[L_iv^i=w^i,\]
	where $w^i$ is the solution to problem \eqref{eq2}.
	
	In what follows, we will prove the following result.
	\begin{pro}\label{31}
Assume that $$\alpha_i-\frac{1}{4}\|L_i\|^2_{L(H_i;L^2(Q))}>0\quad (i=1,2).$$ Then for each $f\in L^2(\O\times(0,T))$, there exists a unique Nash equilibrium $\left(\bar v^1(f),\bar v^2(f)\right)$. Moreover, there exists a positive constant $C,$ such that
		\begin{equation}\label{Nashe}
			\|\left(\bar v^1(f),\bar v^2(f)\right)\|_{H}\le C\left(\|f\|_{L^2(\O\times(0,T))}+1\right).
		\end{equation}
	\end{pro}
	\begin{proof}
		Let $u$ be the solution to
		\begin{align}
			\begin{cases}
				u_t-\Delta u-\frac{\mu}{|x|^2}u=f1_{\O},& (x,t)\in Q,\\
				u=0,& (x,t)\in \Sigma,\\
				u(x,0)=z^0(x),&x\in \Omega,
			\end{cases}
		\end{align}
		then we can write the solution to (\ref{eqn}) into the form $z=L_1(v ^1)+L_2(v^2)+u$. Accordingly, we can rewrite \eqref{Nash2} in the form
		\begin{align*}
			\iint_{\O_{i,d}\times (0,T)}(L_1\bar v^1+L_2\bar v^2+u-z_{i,d})L_iv^i\,dxdt+\alpha_i\iint_{\O_i\times (0,T)}\bar v^i v^i\,dxdt=0,\quad \forall v^i\in H_i.
		\end{align*}
		Let $L_i^*\in L(L^2(Q);H_i)$ be the adjoint operator of $L_i$, then the above formula is equivalent to 
		\begin{align*}
			\iint_{\O_i\times (0,T)}\left[L_i^*\left((L_1\bar v^1+L_2\bar v^2+u-z_{i,d})1_{\O_{i,d}}\right)+\alpha_i\bar v^i\right]v^i\,dxdt=0,\quad \forall v^i\in H_i.
		\end{align*}
		In other words, $(\bar v^1,\bar v^2)$ is a Nash equilibrium if and only if
		\begin{align}
			L_i^*\left((L_1\bar v^1+L_2\bar v^2)1_{\O_{i,d}}\right)+\alpha_i\bar v^i=L_i^*\left((z_{i,d}-u)1_{\O_{i,d}}\right),\quad i=1,2.\label{Nash}
		\end{align}
		
		Define the operator $A:H\to H$ by
		\begin{align*}
			A(v^1,v^2)=(L_1^*((L_1v^1+L_2v^2)1_{\O_{1,d}})+\alpha_1v^1,L_2^*((L_1v^1+L_2v^2)1_{\O_{2,d}})+\alpha_2v^2),
		\end{align*}
		then we have $A\in L(H),$ since $L_i$ is bounded. Moreover, for any $(v^1,v^2)\in H$, we have
		\begin{align*}
			\left|\left( A(v^1,v^2),(v^1,v^2)\right)_H\right|=&\sum_{i=1}^{2}\left|\iint_{\O_i\times (0,T)}\big [L_i^*((L_1v^1+L_2v^2)1_{\O_{i,d}})v^i+\alpha_i|v^i|^2\big]\,dxdt\right|\\
			=&\sum_{i=1}^{2}\left|\iint_{Q}(L_1v^1+L_2v^2)1_{\O_{i,d}}L_iv^i\,dxdt+\iint_{\O_i\times (0,T)}\alpha_i|v^i|^2\,dxdt\right|\\
			\ge& \sum_{i=1}^{2}\left(\alpha_i\|v^i\|^2_{H_i}+\iint_Q|L_iv^i|^21_{\O_{i,d}}\,dxdt-\iint_{Q}|L_{3-i}v^{3-i}||L_iv^i|1_{\O_{i,d}}\,dxdt\right)\\
			\ge& \sum_{i=1}^{2}\left(\alpha_i\|v^i\|^2_{H_i}-\frac{1}{4}\iint_{\O_{i,d}\times (0,T)}|L_{3-i}v^{3-i}|^2\,dxdt\right)\\
			\ge&\sum_{i=1}^{2}\left(\alpha_i-\frac{1}{4}\|L_i\|_{L(H_i;L^2(Q))}^2\right)\|v^i\|^2_{H_i}.
		\end{align*}
	Let
		$$\delta_i :=\alpha_i-\frac{1}{4}\|L_i\|^2_{L(H_i;L^2(Q))}>0$$
		and denote by $\delta:=\min\{\delta_1,\delta_2\},$ then we have
		\begin{align}
			|(A(v^1,v^2),(v^1,v^2))_H|\ge \delta\|(v^1,v^2)\|^2_H.
		\end{align}
Hence, according to Lax-Milgram theorem,  we conclude that $A:H\to H$ is invertible and for each $F\in H^{\prime}$, there exists exactly one pair $(\bar v^1,\bar v^2),$ such that
		\begin{align}
			\left(A(\bar v^1,\bar v^2),(v^1,v^2)\right)_H=\left\langle F,(v^1,v^2)\right\rangle_{H^{\prime},H},\quad \forall (v^1,v^2)\in H.
		\end{align}
		In particular, if $F$ is given by
		\begin{align*}
			\left\langle F,(v^1,v^2)\right\rangle_{H^{\prime},H}=\Big(\left(L_1^*((z_{1,d}-u)1_{\O_{1,d}}),L_2^*((z_{2,d}-u)1_{\O_{2,d}})\right),(v^1,v^2)\Big)_H,
		\end{align*}
then $F\in H'.$ Therefore,  we obtain the existence and uniqueness of solution to problem \eqref{Nash}.
		
		Moreover, we have $$(\bar v^1,\bar v^2)=A^{-1}\left(L_1^*((z_{1,d}-u)1_{\O_{1,d}}),L_2^*((z_{2,d}-u)1_{\O_{2,d}})\right),$$
which implies that
		\[\|(\bar v^1,\bar v^2)\|_H\le \frac{1}{\delta}\left\|\left(L_1^*((z_{1,d}-u)1_{\O_{1,d}}),L_2^*((z_{2,d}-u)1_{\O_{2,d}})\right)\right\|_H.\]
		By the standard energy estimates,  we conclude that
		\begin{align*}
			\|u\|_{L^2(Q)}\le C(\|f\|_{L^2(\O\times (0,T))}+\|z_0\|_{L^2(\Omega)}).
		\end{align*} 
		Consequently, we obtain
		\begin{align}
			\|(\bar v^1,\bar v^2)\|_H\le C(1+\|f\|_{L^2(\O\times (0,T))}),
		\end{align}
		where $C$ is a positive constant depending on $\delta,$ $T,$ $\|z_0\|_{L^2(\Omega)},$ $ \|z_{1,d}\|_{L^2(\O_{1,d}\times(0,T))},$ $ \|z_{2,d}\|_{L^2(\O_{2,d}\times(0,T))}$.
	\end{proof}
	
	Notice that, from \eqref{Nashe} and the energy estimates, the solution $z$ of \eqref{eqn} associated to $f$ and $(\bar v^1(f),\bar v^2(f))$ satisfies
	\begin{align}\label{ineq}
		\|z\|_{L^2(0,T;\mathcal{M})}+\|z_t\|_{L^2(0,T;\mathcal{M}^{\prime})}\le C\left(\|f\|_{L^2(\O\times (0,T))}+1\right).
	\end{align}
	\subsection{The optimality system}
	In this subsection, we will deduce the optimality system that characterizes the Nash equilibrium for the cost functionals $J_i$.

	Multiplying problem \eqref{eq2} by a function $\phi^i$ and integrating by parts, we have
	\begin{align*}
		\iint_{\O_{i,d}}(z-z_{i,d})w^i\,dxdt=\iint_{\O_i\times (0,T)}v^i\phi^i\,dxdt,\quad \forall v^i\in H_i,
	\end{align*}
	where $\phi^i$ is the solution of the following problem
	\begin{align*}
		\begin{cases}
			-\phi_t^i-\Delta \phi^i-\frac{\mu}{|x|^2}\phi^i=(z-z_{i,d})1_{\O_{i,d}},&(x,t)\in Q,\\
			\phi^i=0,&(x,t)\in \Sigma,\\
			\phi^i(x,T)=0,&x\in \Omega.
		\end{cases}
	\end{align*}
	From (\ref{Nash2}), we conclude that $(\bar v^1,\bar v^2)$ is the Nash equilibrium if and only if
	\begin{align*}
		\iint_{\O_i\times (0,T)}(\phi^i+\alpha_i\bar v^i)v^i\,dxdt=0,\quad \forall v^i\in H_i,
	\end{align*}
which implies that
	\[\bar v^i=-\frac{1}{\alpha_i}\phi^i\Big|_{\O_i\times (0,T)},\quad i=1,2.\]
	Thus, we obtain the following optimality system:
	\begin{align}\label{nueq}
		\begin{cases}
			z_t-\Delta z-\frac{\mu}{|x|^2}z=f1_{\O}-\sum\limits_{i=1}^{2}\frac{1}{\alpha_i}\phi^i1_{\O_i},&(x,t)\in Q,\\
			-\phi^i_t-\Delta \phi^i-\frac{\mu}{|x|^2}\phi^i=(z-z_{i,d})1_{\O_{i,d}},&(x,t)\in Q,\\
			z=0,\phi^i=0,  &(x,t)\in \Sigma,\\
			z(x,0)=z^0(x),\phi^i(x,T)=0,&x\in \Omega.
		\end{cases}
	\end{align}
	
	We claim that if $\alpha_i\ (i=1,2)$ are large enough, there exists a unique solution to problem \eqref{nueq}. Denote by
	\begin{equation}\label{X}
		X=L^2(0,T;\mathcal{M})\cap H^1(0,T;\mathcal{M}^{\prime}),
	\end{equation}
	then we conclude from Lemma \ref{well} that there exists a unique weak solution $(z_w,\phi_w^1,\phi_w^2)\in X\times X\times X$ to problem
	\begin{align}\label{ouhe}
		\begin{cases}
			z_t-\Delta z-\frac{\mu}{|x|^2}z=f1_{\O}-\sum\limits_{i=1}^{2}\frac{1}{\alpha_i}\phi^i1_{\O_i},&(x,t)\in Q,\\
			-\phi^i_t-\Delta \phi^i-\frac{\mu}{|x|^2}\phi^i=(w-z_{i,d})1_{\O_{i,d}},&(x,t)\in Q,\\
			z=0,\phi^i=0,  &(x,t)\in \Sigma,\\
			z(x,0)=z^0(x),\phi^i(x,T)=0,&x\in \Omega
		\end{cases}
	\end{align} 
	for any $w\in L^2(Q)$. Therefore, we can define $S(w):=z_w$. From the standard energy estimate, we deduce that there exists a positive constant $C$ independent of $\alpha_i,$ such that
	\[\|S(w_1)-S(w_2)\|_{L^2(Q)}\le C\left(\frac{1}{\alpha_1}+\frac{1}{\alpha_2}\right)\|w_1-w_2\|_{L^2(Q)},\quad \forall w_1,w_2\in L^2(Q).\]
	Therefore, if $\alpha_i$ are large enough,  then the operator $S:L^2(Q)\to L^2(Q)$ is contractive, which implies that the operator $S$ possesses a unique fixed point. It's obvious that $(z,\phi_z^1,\phi_z^2)$ is the solution to problem \eqref{nueq} if and only if $z$ is the fixed point of $S$. Consequently, there exists exactly one weak solution to problem \eqref{nueq}.
	
	\subsection{Null controllability}
	The main objective of this subsection is to prove the null controllability of problem \eqref{nueq}. Based on the standard controllability-observability duality, the null controllability of problem \eqref{nueq} is reduced to an observability inequality for the following adjoint system given by
	\begin{align}
		\begin{cases}
			-\psi_t-\Delta \psi-\frac{\mu}{|x|^2}\psi=\sum_{i=1}^{2}\gamma^i1_{\O_{i,d}},&(x,t)\in Q,\\
			\gamma^i_t-\Delta \gamma^i-\frac{\mu}{|x|^2}\gamma^i=-\frac{1}{\alpha_i}\psi1_{\O_i} ,&(x,t)\in Q,\\
			\psi=0,\gamma^i=0, &(x,t)\in \Sigma,\\
			\psi(x,T)=\psi^T(x),\gamma^i(x,0)=0,&x\in\Omega.
		\end{cases}\label{adeq}
	\end{align}
	Thus, we have to prove the following observability inequality.
	\begin{pro}\label{pro}
		Suppose that $\alpha_i\ (i=1,2)$ are large enough.
		\begin{enumerate}[\rm{(}1\rm{)}]
			\item Assume that 
			\begin{align}\label{con1}
					\O_{1,d}=\O_{2,d},\,\,\,\,y_{1,d}=y_{2,d}
			\end{align}
			(and in this case we set $\O_d:=\O_{i,d}$ and $y_{d}:=y_{i,d}$), then there exists a constant $C>0,$ such that for any $\psi^T\in L^2(\Omega)$, the solution $(\psi,\gamma^1,\gamma^2)$ to problem \eqref{adeq} satisfies
			\begin{equation}\label{ob1}
				\|\psi(0)\|^2+\iint_{\O_{d}\times(0,T)}\rho^{-2}|\gamma^1+\gamma^2|^2\,dxdt\le C\iint_{\O\times(0,T)}|\psi|^2\,dxdt,
			\end{equation}
			where $\rho$ is defined as in \eqref{rho}.
			\item Assume that 
			\begin{equation}\label{con2}
				\O_{1,d}\cap\O\neq\O_{2,d}\cap\O,
			\end{equation}
			then a similar property holds with \eqref{ob1} replaced by
			\begin{align}\label{ob2} \|\psi(0)\|^2+\sum_{i=1}^{2}\iint_{\O_{i,d}\times(0,T)}\rho^{-2}|\gamma^i|^2\,dxdt\le C\iint_{\O\times(0,T)}|\psi|^2\,dxdt,
			\end{align}
			where $\rho$ is defined as in \eqref{rho3}.
		\end{enumerate}
	\end{pro}
	\begin{proof}
		We need to distinguish the proof into two cases.
		
		\noindent\textbf{Case 1:} We assume that \eqref{con1} holds.
		
		Let $s,\ \lambda,\ \Psi(x),\ \Phi(x),\ \theta(t)$ and $\sigma(x,t)$ are defined as in Section \ref{preliminary}. Obviously,  for any sufficiently large $\lambda,$ we have
		\begin{align*}
				\sigma(x,t)>0,\,\,\,(x,t)\in Q;
				\lim_{t\to 0^+}\sigma(x,t)=\lim_{t\to T^-}\sigma(x,t)=+\infty,\,\,\,x\in\Omega.
		\end{align*}
		Let $\omega_0$ be given as in Lemma \ref{lem1} satisfying  $\omega_0\subset \subset \widetilde{\omega}:=\O\cap \O_{d}$ and define a smooth cut-off function $\zeta$ on $\Omega$ by
		\[
		\begin{cases}
			\zeta=1,&x \in \omega_0,\\
			\zeta=0,&x \in \Omega\setminus \widetilde{\omega},\\
			\zeta\ge 0,&x\in \Omega.
		\end{cases}
		\] 
Denote by $\bar \gamma:=\gamma^1+\gamma^2$, then $\bar \gamma$ be a solution of the following problem
		\begin{align}\label{bar}
			\begin{cases}
				\bar\gamma_t-\Delta \bar\gamma-\frac{\mu}{|x|^2}\bar\gamma=-\sum_{i=1}^{2}\frac{1}{\alpha_i}\psi1_{\O_i},&(x,t)\in Q,\\
				\bar\gamma=0, &(x,t)\in \Sigma,\\
				\bar\gamma(x,0)=0,&x\in\Omega.
			\end{cases}
		\end{align}
		Thanks to $\theta(T-t)=\theta(t)$, we conclude from Lemma \ref{carle} that
		\begin{align*}
			s\iint_{Q}\theta e^{-2\sigma}|\psi|^2\,dxdt\le C\left(s^3\lambda^4\iint_{\omega_0\times(0,T)}\theta^3\Phi^3e^{-2\sigma}|\psi|^2\,dxdt+\iint_{\O_d}e^{-2\sigma}|\bar\gamma|^2\,dxdt\right)
		\end{align*}
		and
		\begin{align*}
			&s\iint_{Q}\theta e^{-2\sigma}|\bar\gamma|^2\,dxdt+s\lambda^2\iint_{\widetilde{\Omega}\times(0,T)}\theta \Phi e^{-2\sigma}|\nabla \bar\gamma|^2\,dxdt\\\le &C\left(s^3\lambda^4\iint_{\omega_0\times (0,T)}\theta^3\Phi^3e^{-2\sigma}|\bar\gamma|^2\,dxdt+\iint_{Q}e^{-2\sigma}\left|\sum_{i=1}^{2}\frac{1}{\alpha_i}\psi1_{\O_i}\right|^2\,dxdt\right).
		\end{align*}
From the above two inequalities, we deduce that
		\begin{align}\label{two}
			&s\iint_{Q}\theta e^{-2\sigma}|\psi|^2\,dxdt+s\iint_Q\theta e^{-2\sigma}|\bar\gamma|^2\,dxdt+s\lambda^2\iint_{\widetilde{\Omega}\times(0,T)}\theta\Phi e^{-2\sigma}|\nabla \bar\gamma|^2\,dxdt\notag\\
			\le &C\left(\sum_{i=1}^{2}\iint_{\O_i\times(0,T)}e^{-2\sigma}\frac{|\psi|^2}{\alpha_i^2}\,dxdt+\iint_{\O_d\times(0,T)}e^{-2\sigma}|\bar\gamma|^2\,dxdt\right)\notag\\
			&+C\left(s^3\lambda^4\iint_{\omega_0\times(0,T)}\theta^3\Phi^3e^{-2\sigma}|\psi|^2\,dxdt+s^3\lambda^4\iint_{\omega_0\times (0,T)}\theta^3\Phi^3e^{-2\sigma}|\bar\gamma|^2\,dxdt\right).
		\end{align}
In what follows, we need to estimate the last term in \eqref{two}.  To do this,  let $p:=\theta^3\Phi^3e^{-2\sigma}$, we observe that $\theta^{-\frac{1}{2}}e^{\sigma}p$, $\theta^{-\frac{1}{2}}e^{\sigma}|p_t|$, $\theta^{-\frac{1}{2}}e^{\sigma}|\nabla p|$ and $\theta^{-\frac{1}{2}}e^{\sigma}|\Delta p|$ are bounded. Hence
		\begin{equation}\label{11}
		\begin{aligned}
			&s^3\lambda^4\iint_{\omega_0\times(0,T)}p|\bar\gamma|^2\,dxdt
			\le s^3\lambda^4\iint_{\O_d}\zeta p|\bar\gamma|^2\,dxdt\\
			\le &s^3\lambda^4\iint_{Q}\zeta p\bar\gamma(-\psi_t-\Delta \psi-\frac{\mu}{|x|^2}\psi)\,dxdt
			\\
			\le
			&Cs^3\lambda^4\iint_{\O\times(0,T)}p\left(\frac{1}{\alpha_1}+\frac{1}{\alpha_2}\right)|\psi|^2\,dxdt+Cs^3\lambda^4\iint_{\widetilde{\omega}\times(0,T)}\left(p+|\nabla p|\right)|\nabla \bar\gamma||\psi|\,dxdt\\&+Cs^3\lambda^4\iint_{\widetilde{\omega}\times(0,T)}\left(p+|p_t|+|\nabla p|+|\Delta p|\right)|\bar\gamma||\psi|\,dxdt\\
			\le
			&\iint_{\widetilde{\omega}\times(0,T)}\theta e^{-2\sigma}|\bar\gamma|^2\,dxdt+\iint_{\widetilde{\omega}\times(0,T)}\theta\Phi e^{-2\sigma}|\nabla\bar\gamma|^2)\,dxdt+Cs^6\lambda^8\iint_{\O\times(0,T)}|\psi|^2\,dxdt.
		\end{aligned}
		\end{equation}		
		\if{Next we calculate that
		\begin{align*}
			\iint_{Q}\nabla(\zeta^2p\bar\gamma)\cdot\nabla \psi\,dxdt=-\iint_{Q}&\bigg(2|\nabla\zeta|^2p\bar\gamma\psi+2\zeta\Delta \zeta p\bar\gamma\psi+4\zeta\bar\gamma\psi\nabla \zeta\cdot \nabla p+4\zeta p\psi\nabla \zeta \cdot \nabla \bar\gamma\\
			&+\zeta^2\Delta p\bar\gamma \psi +2\zeta^2\psi\nabla p\cdot \nabla \bar\gamma+\zeta^2\psi \nabla p\cdot\nabla \bar\gamma+\zeta^2p\Delta \bar\gamma\psi\bigg)\,dxdt,
		\end{align*}
		then by \eqref{bar}, we see that
		\begin{align}\label{iden}
				\iint_{Q}&\left(\zeta^2p_t\bar\gamma\psi+\zeta^2p\bar\gamma_t\psi+\nabla(\zeta^2p\bar\gamma)\cdot\nabla \psi-\frac{\mu}{|x|^2}\zeta^2p\psi \bar\gamma\right)\,dxdt\notag\\=\iint_{Q}&\bigg(\zeta^2p_t\bar\gamma\psi-\big(\frac{1}{\alpha_1}1_{\O_1}+\frac{1}{\alpha_2}1_{\O_2}\big)\zeta^2p|\psi|^2-2|\nabla\zeta|^2p\bar\gamma\psi-2\zeta\Delta \zeta p\bar\gamma\psi-\zeta^2\psi \nabla p\cdot\nabla \bar\gamma\notag\\ &-4\zeta\bar\gamma\psi\nabla \zeta\cdot \nabla p-4\zeta p\psi\nabla \zeta \cdot \nabla \bar\gamma-\zeta^2\Delta p\bar\gamma \psi -2\zeta^2\psi\nabla p\cdot \nabla \bar\gamma\bigg)\,dxdt. 
		\end{align}
		Since $\zeta,\  \nabla\zeta,\  \Delta\zeta,\  p,\ p_t,\ \nabla p$ and $\Delta p$ are bounded functions, we 
		
		\begin{equation}
			\begin{aligned}
				\sum_{i=1}^{2}\iint_{\omega_i\times (0,t)}\theta^3\Phi_i^3e^{-2\sigma_i}|\gamma^i|^2\,dxdt\le
				C_\epsilon s^6\lambda^8\iint_{\O\times(0,T)}|\psi|^2\,dxdt\\+\epsilon\sum_{i=1}^{2}\left(\iint_{\O_{i,d}\times(0,T)}p^i|\gamma^i|^2\,dxdt+\iint_{\O_{i,d}\times (0,T)}p^i|\nabla \gamma^i|^2\,dxdt\right).\label{ga}
			\end{aligned}
		\end{equation}
		Combining (\ref{car}) and (\ref{ga}) we have
		\begin{align*}
			s\iint_{Q}\theta e^{-2\sigma_1}|\psi|^2\,dxdt+\sum_{i=1}^{2}s\iint_Q\theta e^{-2\sigma_i}|\gamma^i|^2\,dxdt\le C(s,\lambda)\iint_{\O\times(0,T)}|\psi|^2\,dxdt.
		\end{align*}}\fi
		Notice that $\O_d\cap\overline{B(0,1)}=\emptyset$, then $\widetilde{\omega}\subset\widetilde{\Omega}$. Combining inequality \eqref{two} with inequality \eqref{11}, we deduce
		\begin{equation}\label{Opsi}
			s\iint_{Q}\theta e^{-2\sigma}|\psi|^2\,dxdt+s\iint_Q\theta e^{-2\sigma}|\bar\gamma|^2\,dxdt\le Cs^6\lambda^8\iint_{\O\times(0,T)}|\psi|^2\,dxdt
		\end{equation}
for any sufficiently large $\alpha_i$ and $s.$

		Taking the inner product in $L^2(\Omega)$ of the first equation of problem (\ref{adeq}) with $\psi$, we obtain
		\begin{align*}
			-\frac{1}{2}\frac{d}{dt}\|\psi(t)\|^2+\int_{\Omega}\left(|\nabla\psi|^2-\frac{\mu}{|x|^2}|\psi|^2\right)\,dx=(\psi,\bar\gamma1_{\O_d}).
		\end{align*}
		For any $r,\ s\in [0,\frac{3T}{4}]$ with $r<s$, we apply Hardy inequality and H\"{o}lder inequality to yield
		\begin{align*}
			\|\psi(r)\|^2\le \|\psi(s)\|^2+\int_{r}^{s}\int_{\Omega}|\psi|^2\,dxdt+\int_{r}^{s}\int_{\O_d}|\bar \gamma|^2\,dxdt.
		\end{align*}
		Then the classical Gronwall inequality implies
		\begin{align*}
			\|\psi(r)\|^2\le C\left(\|\psi(s)\|^2+\int^{\frac{3T}{4}}_{0}\int_{\O_d}|\bar\gamma|^2\,dxdt\right).
		\end{align*}
		Integrating the above inequality over $[\frac{T}{4},\frac{3T}{4}]$ with respect to $s$, we have
		\begin{align}\label{54}
			\|\psi(r)\|^2\le C\left(\int_{\frac{T}{4}}^{\frac{3T}{4}}\int_{\Omega}|\psi|^2\,dxdt+\int^{\frac{3T}{4}}_0\int_{\O_d}|\bar\gamma|^2\,dxdt\right)
		\end{align}
for any $r\le\frac{T}{4}.$	
		
It follows from the standard energy methods and inequality \eqref{54} that
		\begin{align*}
			\int_{0}^{\frac{T}{4}}\int_{\Omega}|\bar\gamma|^2\,dxdt&\le C\left(\frac{1}{\alpha_1^2}+\frac{1}{\alpha_2^2}\right)\int_{0}^{\frac{T}{4}}\int_{\Omega}|\psi|^2\,dxdt\\
			&\le C\left(\frac{1}{\alpha_1^2}+\frac{1}{\alpha_2^2}\right)\left(\int_{\frac{T}{4}}^{\frac{3T}{4}}\int_{\Omega}|\psi|^2\,dxdt+\int^{\frac{3T}{4}}_0\int_{\O_d}|\bar\gamma|^2\,dxdt\right).
		\end{align*}
Letting $\alpha_i\ (i=1,2)$ be sufficiently large, we have
		\begin{align}\label{53}
			\int_{0}^{\frac{T}{4}}\int_{\O_d}|\bar\gamma|^2\,dxdt\le C\left(\int_{\frac{T}{4}}^{\frac{3T}{4}}\int_{\Omega}|\psi|^2\,dxdt+\int_{\frac{T}{4}}^{\frac{3T}{4}}\int_{\O_d}|\bar\gamma|^2\,dxdt\right).
		\end{align}
		Therefore,  we conclude from inequalities \eqref{54} -\eqref{53} and $r=0$ that
		\begin{align}\label{psi0}
			\|\psi(0)\|^2\le C\left(\int_{\frac{T}{4}}^{\frac{3T}{4}}\int_{\Omega}|\psi|^2dxdt+\int_{\frac{T}{4}}^{\frac{3T}{4}}\int_{\O_d}|\bar\gamma|^2\,dxdt\right).
		\end{align}
		Notice that $$\min_{(x,t)\in \Omega\times[\frac{T}{4},\frac{3T}{4}]}\theta e^{-2\sigma}>0$$ and define
		\begin{align}\label{rho}
			\rho:=e^{\sigma}\theta^{-\frac{1}{2}},
		\end{align}
		it follows from inequality \eqref{Opsi} and inequality \eqref{psi0} that
		\begin{equation*}
			\begin{split}
				\|\psi(0)\|^2+\iint_{\O_d\times(0,T)}\rho^{-2}|\bar\gamma|^2\,dxdt&\le C\left(\int_{\frac{T}{4}}^{\frac{3T}{4}}\int_{\Omega}|\psi|^2dxdt+\iint_{\O_d\times(0,T)}\theta e^{-2\sigma}|\bar\gamma|^2\,dxdt\right)\\
				&\le C\iint_{\O\times(0,T)}|\psi|^2\,dxdt.
			\end{split}
		\end{equation*}
\noindent\textbf{Case 2:} Now, we assume that \eqref{con2} holds.

Let $\widetilde{\O}\subset\subset\O$ be a nonempty open set such that $\O_{i,d}\cap\widetilde{O}\neq\emptyset$ for $i=1,2$. Then either 
\begin{equation}\label{O1}
	(\O_{1,d}\cap\widetilde{O})\setminus\O_{2,d}\neq\emptyset\quad \text{and}\quad (\O_{2,d}\cap\widetilde{O})\setminus\O_{1,d}\neq\emptyset
\end{equation}
or 
\begin{equation}\label{O2}
	\O_{i,d}\cap\widetilde{O}\subset\O_{3-i,d},\quad i=1\ \text{or}\ 2.
\end{equation}

1. If \eqref{O1} holds, then there exist some nonempty open subsets $\omega_i$ and $\widetilde{\omega}_i$ satisfying $\omega_i\subset\subset\widetilde{\omega}_i\subset\subset\widetilde{\O}\cap\O_{i,d}$ with $\widetilde{\omega}_{i}\cap\O_{3-i,d}=\emptyset$ for $i=1,2$. From Lemma \ref{lem1}, we conclude that there exist two smooth functions $\Psi_i(x)$ satisfying
\begin{align}\label{weight1}
	\begin{cases}
		\Psi_i(x)=\ln(|x|),& \quad x\in B(0,1),\\
		\Psi_i(x)=0, &\quad x\in \Gamma,\\
		\Psi_i(x)> 0,&\quad x\in \widetilde{\Omega},\\
		|\nabla \Psi_i|\ge \delta,&\quad x\in \overline{\Omega} \setminus \omega_i,\\
		\Psi_1(x)=\Psi_2(x),&\quad x\in\Omega\setminus\widetilde{\O}
	\end{cases}
\end{align}
and 
\begin{equation}\label{weight2}
	\sup_{x\in\Omega} \Psi_1=\sup_{x\in \Omega} \Psi_2.
\end{equation}
As in Section \ref{preliminary}, we introduce the weight functions
 \begin{equation}\label{weight3}
 	\theta(t)=t^{-3}(T-t)^{-3},\quad \sigma_i(x,t)=s\theta(t)(e^{2\lambda\sup \Psi_i}-\frac{1}{2}|x|^2-e^{\lambda\Psi_i}),\quad \Phi_i(x)=e^{\lambda \Psi_i(x)},
 \end{equation} 
 where $s$ and $\lambda$ are positive constants. Furthermore, when $\lambda$ is large enough we have
\begin{align*}
	\begin{cases}
		\sigma_i(x,t)>0,&(x,t)\in Q,\\
		\lim_{t\to 0^+}\sigma_i(x,t)=\lim_{t\to T^-}\sigma_i(x,t)=+\infty,&x\in\Omega.
	\end{cases}
\end{align*}
Assume that the open sets $\O_0,\ \widetilde{\O}_1$ and $\widetilde{\O}_2$ satisfy $\widetilde{\O}\subset\subset\widetilde{\O}_1\subset\subset\widetilde{\O}_2\subset\subset\O_0\subset\subset\O$ and $0\notin \overline{\O}_0\setminus \widetilde{\O}.$ Let $\hat\eta$ be a cut-off function satisfying
\begin{equation}\label{eta}
	\begin{cases}
		\hat\eta=1,&x \in \Omega\setminus\widetilde{\O}_2,\\
		\hat\eta=0,&x \in \widetilde{\O}_1,\\
		0\le \hat\eta\le 1,&x\in \Omega.
	\end{cases}
\end{equation}
Then $\hat{\psi}=\hat\eta \psi$ is the solution to problem
\begin{align}
	\begin{cases}
		-\hat\psi_t-\Delta \hat\psi-\frac{\mu}{|x|^2}\hat\psi=\sum_{i=1}^{2}\hat\eta\gamma^i1_{\O_{i,d}}-2\nabla\hat\eta\cdot\nabla\psi-\psi\Delta\hat\eta,&(x,t)\in Q,\\
		\psi=0, &(x,t)\in \Sigma,\\
		\psi(x,T)=\hat\zeta\psi^T,&x\in\Omega.
	\end{cases}
\end{align}
From Lemma \ref{carle}, we conclude that
\begin{align}\label{332}
	s\iint_{Q}\theta e^{-2\sigma_i}|\hat\psi|^2\,dxdt
	\le& C\iint_{Q}e^{-2\sigma_i}\bigg(\sum_{j=1}^{2}\hat\eta\gamma^j1_{\O_{j,d}}\bigg)^2\,dxdt+C\iint_{Q}e^{-2\sigma_i}|\nabla\hat\eta|^2|\nabla \psi|^2\,dxdt\notag\\&+C\iint_{Q}e^{-2\sigma_i}|\Delta\hat\eta|^2|\psi|^2\,dxdt+Cs^3\lambda^4\iint_{\omega_i\times(0,T)}\theta^3\Phi_1^3e^{-2\sigma_i}|\hat\psi|^2\,dxdt
\end{align}
and
\begin{align}\label{333}
	&s\iint_{Q}\theta e^{-2\sigma_i}|\gamma^i|^2\,dxdt+s\lambda^2\iint_{\widetilde{\Omega}\times (0,T)}\theta \Phi_ie^{-2\sigma_i}|\nabla \gamma^i|^2\,dxdt\notag\\
	\le &C\left(s^3\lambda^4\iint_{\omega_i\times (0,T)}\theta^3\Phi_i^3e^{-2\sigma_i}|\gamma^i|^2\,dxdt+\iint_{Q}e^{-2\sigma_i}\frac{|\psi|^2}{\alpha_i^2}1_{\O_i}\,dxdt\right).
\end{align}
By virtue of inequalities \eqref{332}-\eqref{333} and 
\[\iint_{Q}\theta e^{-2\sigma_i}|\hat\psi|^2\,dxdt\ge \iint_Q\theta e^{-2\sigma_i}|\psi|^2\,dxdt-C\iint_{\O\times(0,T)}|\psi|^2\,dxdt,\]
we obtain
\begin{align}\label{car}
	&\sum_{i=1}^{2}\left(s\iint_{Q}\theta e^{-2\sigma_i}|\psi|^2\,dxdt+s\iint_Q\theta e^{-2\sigma_i}|\gamma^i|^2\,dxdt+s\lambda^2\iint_{\widetilde{\Omega}\times(0,T)}\theta\Phi_ie^{-2\sigma_i}|\nabla \gamma^i|^2\,dxdt\right)\notag\\
	\le &C\Bigg( s^3\lambda^4\iint_{\O\times(0,T)}|\psi|^2\,dxdt+\sum_{i=1}^{2}\bigg( \iint_{Q}e^{-2\sigma_i}|\nabla \hat\eta|^2|\nabla\psi|^2\,dxdt+\iint_{Q}e^{-2\sigma_i}\Big(\sum_{j=1}^{2}\hat\eta\gamma^j1_{\O_{j,d}}\Big)^2\,dxdt\bigg)\Bigg) \notag\\
	&+C\sum_{i=1}^{2} \left(\iint_{Q}e^{-2\sigma_i}\frac{|\psi|^2}{\alpha_i^2}\,dxdt+s^3\lambda^4\iint_{\omega_i\times(0,T)}\theta^3\Phi_i^3e^{-2\sigma_i}|\gamma^i|^2\,dxdt\right) .
\end{align}
It follows from Lemma \ref{Cap} that
\begin{equation}\label{cap}
	\begin{aligned}
		&\iint_{Q}e^{-2\sigma_i}|\nabla \hat\eta|^2|\nabla\psi|^2\,dxdt\le C\iint_{(\widetilde{\O}_2\setminus\widetilde\O_1)\times(0,T)}e^{-2\sigma_i}|\nabla\psi|^2\,dxdt\\
		\le &C\iint_{(\O_0\setminus\widetilde\O)\times(0,T)}e^{-\sigma_i}|\psi|^2\,dxdt+\iint_{(\O_0\setminus\widetilde\O)\times(0,T)}e^{-2\sigma_i}|\sum_{j=1}^2\gamma^j1_{\O_{j,d}}|^2\,dxdt.
\end{aligned}
\end{equation}
Notice that $$\sigma_1(x,t)=\sigma_2(x,t),\quad \text{for}\ x\in \Omega\setminus\widetilde{\O},\ t\in (0,T),$$
which implies that
\begin{equation}\label{gam}
	\iint_{Q}e^{-2\sigma_i}\Big(\sum_{j=1}^{2}\hat\eta\gamma^j1_{\O_{j,d}}\Big)^2\,dxdt\le C\sum_{i=1}^{2}\iint_{Q}e^{-2\sigma_i}|\gamma^i|^2\,dxdt
\end{equation}
and
\begin{equation}
	\iint_{(\O_0\setminus\widetilde\O)\times(0,T)}e^{-2\sigma_i}|\sum_{j=1}^2\gamma^j1_{\O_{j,d}}|^2\,dxdt\le C\sum_{i=1}^{2}\iint_{Q}e^{-2\sigma_i}|\gamma^i|^2\,dxdt.
\end{equation}
To eliminate the last term in the right-hand side in \eqref{car}, we also define the smooth function $\zeta_i$ on $\Omega$ as follows:
\[
\begin{cases}
	\zeta_i=1,&x \in \omega_i,\\
	\zeta_i=0,&x \in \Omega\setminus \widetilde{\omega}_i,\\
	\zeta_i\ge 0,&x\in \Omega.
\end{cases}
\]
In view of $\widetilde{\omega}_i\cap\O_{3-i,d}=\emptyset$, we can deduce
	\begin{align}\label{p}
		\nonumber&s^3\lambda^4\sum_{i=1}^{2}\iint_{\omega_i\times(0,T)}p^i|\gamma^i|^2\,dxdt
		\nonumber\le s^3\lambda^4\sum_{i=1}^{2}\iint_{Q}\zeta_ip^i|\gamma^i|^2\,dxdt\\
		\nonumber\le &s^3\lambda^4\sum_{i=1}^{2}\iint_{Q}\zeta_ip^i\gamma^i(-\psi_t-\Delta \psi-\frac{\mu}{|x|^2}\psi)\,dxdt\\
		\nonumber\le &s^3\lambda^4\sum_{i=1}^{2}\iint_{\O\times(0,T)}\frac{p^i}{\alpha_i}|\psi|^2\,dxdt+s^3\lambda^4\sum_{i=1}^{2}\iint_{\widetilde{\omega}_i\times(0,T)}(p^i+|\nabla p^i|)|\nabla \gamma^i||\psi|\,dxdt\\
		\nonumber&+s^3\lambda^4\sum_{i=1}^{2}\iint_{\widetilde{\omega}_i\times(0,T)}(p^i+|p^i_t|+|\nabla p^i|+|\Delta p^i|)|\psi||\gamma^i|\,dxdt\\
		\nonumber\le&\sum_{i=1}^{2}\iint_{\widetilde{\omega}_i\times(0,T)}\theta e^{-2\sigma_i}|\gamma^i|^2\,dxdt+\sum_{i=1}^{2}\iint_{\widetilde{\omega}_i\times(0,T)}\theta\Phi_i e^{-2\sigma_i}|\nabla\gamma^i|^2\,dxdt\\&+Cs^6\lambda^8\iint_{\O\times(0,T)}|\psi|^2\,dxdt,
	\end{align}
where $p^i=\theta^3\Phi_i^3e^{-2\sigma_i}$.

By taking $\alpha_i,$ $s$ large enough and combining \eqref{car}\textendash\eqref{p}, we obtain
\begin{align}\label{12}
	\sum_{i=1}^{2}\left(s\iint_{Q}\theta e^{-2\sigma_i}|\psi|^2\,dxdt+s\iint_Q\theta e^{-2\sigma_i}|\gamma^i|^2\,dxdt\right)\le Cs^6\lambda^8\iint_{\O\times(0,T)}|\psi|^2\,dxdt.
\end{align}
By carrying out the similar proof of inequality \eqref{psi0}, we deduce that
\begin{align}\label{13}
	\|\psi(0)\|^2\le C\left(\int_{\frac{T}{4}}^{\frac{3T}{4}}\int_{\Omega}|\psi|^2\,dxdt+\sum_{i=1}^{2}\int_{\frac{T}{4}}^{\frac{3T}{4}}\int_{\O_{i,d}}|\gamma^i|^2\,dxdt\right).
\end{align}
Define
\begin{equation}\label{rho3}
	\rho:=\max\{e^{\sigma_1}\theta^{-\frac{1}{2}},e^{\sigma_2}\theta^{-\frac{1}{2}}\},
\end{equation}
we infer from inequalities \eqref{12}-\eqref{13} that
\begin{equation*}
	\|\psi(0)\|^2+\sum_{i=1}^{2}\iint_{\O_{i,d}\times(0,T)}\rho^{-2}|\gamma^i|^2\,dxdt\le C\iint_{\O\times(0,T)}|\psi|^2\,dxdt.
\end{equation*}

2. Assume that \eqref{O2} holds. Without loss of generality, we can assume that $\O_{1,d}\cap\widetilde{\O}\subset\O_{2,d}$.
Then there exist the nonempty open sets $\omega_i$ and $\widetilde{\omega}_i$ satisfying $\omega_i\subset\subset\widetilde{\omega}_i\subset\subset\widetilde{\O}\cap\O_{i,d}$ with $\widetilde{\omega}_{2}\cap\O_{1,d}=\emptyset$. Let $\Psi_i$, $\theta$, $\sigma_i$, $\Phi_i,$  $\hat\eta$ be given as in \eqref{weight1}-\eqref{eta} and denote by $\overline{\gamma}:=\gamma^1+\gamma^2$.  By carrying out the similar proof of inequality \eqref{two}, \eqref{car} and \eqref{cap}, we conclude that
\begin{align}\label{1351}
	&\sum_{i=1}^{2}s\iint_{Q}\theta e^{-2\sigma_i}|\psi|^2\,dxdt+s\iint_{Q}\theta e^{-2\sigma_1}|\bar\gamma|^2\,dxdt+s\lambda^2\iint_{\widetilde{\Omega}\times(0,T)}\theta\Phi_1 e^{-2\sigma_1}|\nabla \bar\gamma|^2\,dxdt\notag\\
	&+3s\iint_Q\theta e^{-2\sigma_2}|\gamma^2|^2\,dxdt+3s\lambda^2\iint_{\widetilde{\Omega}\times(0,T)}\theta\Phi_2e^{-2\sigma_2}|\nabla \gamma^2|^2\,dxdt\notag\\
	\le &C\Bigg( s^3\lambda^4\iint_{\O\times(0,T)}|\psi|^2\,dxdt+\sum_{i=1}^{2}\iint_{Q}e^{-2\sigma_i}\Big(\sum_{j=1}^{2}\hat\eta\gamma^j1_{\O_{j,d}}\Big)^2\,dxdt\Bigg) \notag\\
	&+C\left(\iint_{Q}e^{-2\sigma_1}\sum_{i=1}^{2}\frac{|\psi|^2}{\alpha_i^2}\,dxdt+s^3\lambda^4\iint_{\omega_1\times (0,T)}\theta^3\Phi_1^3e^{-2\sigma_1}|\bar\gamma|^2\,dxdt\right)\notag\\
	&+C\left( \iint_{Q}e^{-2\sigma_2}\frac{|\psi|^2}{\alpha_2^2}\,dxdt+s^3\lambda^4\iint_{\omega_2\times(0,T)}\theta^3\Phi_2^3e^{-2\sigma_2}|\gamma^2|^2\,dxdt\right).
\end{align}
Arguing as in the proof of inequality \eqref{gam}, we obtain
\begin{equation}
	\sum_{i=1}^{2}\iint_{Q}e^{-2\sigma_i}\Big(\sum_{j=1}^{2}\hat\eta\gamma^j1_{\O_{j,d}}\Big)^2\,dxdt\le C\iint_Qe^{-2\sigma_1}|\bar\gamma|^2\,dxdt+C\iint_Qe^{-2\sigma_2}|\gamma^2|^2\,dxdt.
\end{equation}
By performing the similar proof of \eqref{11} and \eqref{p}, we deduce
\begin{align}\label{515}
	&\iint_{\omega_1\times(0,T)}\theta^3\Phi_1^3e^{-2\sigma_1}|\bar\gamma|^2\,dxdt+s^3\lambda^4\iint_{\omega_2\times(0,T)}\theta^3\Phi_2^3e^{-2\sigma_2}|\gamma^2|^2\,dxdt\notag\\
	\le& Cs^6\lambda^8\iint_{\O\times(0,T)}|\psi|^2\,dxdt+\iint_{\widetilde{\omega}_1\times(0,T)}\theta e^{-2\sigma_1}|\bar\gamma|^2\,dxdt+\iint_{\widetilde{\omega}_1\times(0,T)}\theta\Phi_1 e^{-2\sigma_1}|\nabla\bar\gamma|^2\,dxdt\notag\\
	&+\iint_{\widetilde{\omega}_2\times(0,T)}\theta e^{-2\sigma_2}|\gamma^2|^2\,dxdt+\iint_{\widetilde{\omega}_2\times(0,T)}\theta\Phi_2 e^{-2\sigma_2}|\nabla\gamma^2|^2\,dxdt.
\end{align} 
Thus, for sufficiently large $\alpha_i$ and $s$, it follows from \eqref{1351}\textendash\eqref{515} that 
\begin{align}\label{14}
	\sum_{i=1}^{2}\left(s\iint_{Q}\rho^{-2}|\psi|^2\,dxdt+s\iint_Q\rho^{-2}|\gamma^i|^2\,dxdt\right)\le Cs^6\lambda^8\iint_{\O\times(0,T)}|\psi|^2\,dxdt,
\end{align}
where $\rho(x,t)$ is defined as \eqref{rho3}.  

Along with \eqref{13} -\eqref{14}, we obtain
\begin{equation*}
	\|\psi(0)\|^2+\sum_{i=1}^{2}\iint_{\O_{i,d}\times(0,T)}\rho^{-2}|\gamma^i|^2\,dxdt\le C\iint_{\O\times(0,T)}|\psi|^2\,dxdt.
\end{equation*}
\end{proof}

In the sequel, we will consider problem \eqref{eq1} with $F\equiv 0$ and prove the following result.
\begin{thm}\label{th1}
	Assume that either \eqref{con1} or \eqref{con2} holds, $F\equiv 0$ and the constants $\alpha_i\ (i=1,2)$ are large enough. If $\bar y$ is the unique solution to problem \eqref{tr} with the initial state $\bar y^0 \in L^2(\Omega)$ satisfying
	\begin{align}\label{ass}
		\iint_{\O_{i,d}\times(0,T)}\rho^2|\bar{y}-y_{i,d}|^2\,dxdt< \infty,\quad i=1,2,
	\end{align}
	where the weight function $\rho$ is the same as in Proposition \ref{pro}. Then for any $y^0\in L^2(\Omega)$, there exists a control $\hat f\in L^2(\O\times(0,T))$ and an associated Nash equilibrium $(\bar{v}^1,\bar{v}^2),$ such that the solution to \textup{(\ref{eq1})} satisfies \eqref{exa}. Moreover, $\hat f$ is the unique solution to the extremal problem \eqref{op}\textendash\eqref{exa}.
\end{thm}
\begin{proof}
	Combining (\ref{nueq}) and (\ref{adeq}), we see that
	\begin{align}\label{formula}
		\iint_{\O\times(0,T)}f\psi \,dxdt=\int_{\Omega}z(T)\psi^T\,dx-\int_{\Omega}z^0\psi(0)\,dx+\sum_{i=1}^{2}\iint_{\O_{i,d}\times (0,T)}z_{i,d}\gamma^i\,dxdt.
	\end{align}
	Note that the null controllability of \eqref{nueq} holds if and only if for each $z^0\in L^2(\Omega)$ , there exists a $f\in L^2(\O\times (0,T)),$ such that
	\begin{align}\label{nu}
		\iint_{\O\times(0,T)}f\psi \,dxdt=-\int_{\Omega}z^0\psi(0)\,dx+\sum_{i=1}^{2}\iint_{\O_{i,d}\times (0,T)}z_{i,d}\gamma^i\,dxdt,\quad \forall \psi^T\in L^2(\Omega).
	\end{align}
	
	For each $z^0\in L^2(\Omega)$ and any $\epsilon>0$, we introduce  the functional $F_{\epsilon}:L^2(\Omega)\to \R$ defined by
	\begin{align*}
		F_{\epsilon}(\psi^T)=\frac{1}{2}\iint_{\O\times (0,T)}|\psi|^2\,dxdt+\epsilon \|\psi^T\|+\left(z^0,\psi(0)\right)-\sum_{i=1}^{2}\iint_{\O_{i,d}\times (0,T)}z_{i,d}\gamma^i\,dxdt.
	\end{align*}
	From the usual energy estimate and the linearity of equation \eqref{nueq}, we could obtain that $F_{\epsilon}$ is continuous, strictly convex and differentiable except $0$. As a consequence of Proposition \ref{pro}, $F_{\epsilon}$ is also coercive in $L^2(\Omega)$. Thus, it possesses a unique minimizer $\psi^T_{\epsilon}$. Let us denote by $(\psi_{\epsilon},\gamma^1_{\epsilon},\gamma^2_{\epsilon})$ the solution of (\ref{adeq}) associated with $\psi_{\epsilon}^T$.
	
	Let $f_{\epsilon}=\psi_{\epsilon} 1_{\O\times(0,T)}$ and denote by $(z_{\epsilon},\phi^1_{\epsilon},\phi^2_{\epsilon})$ the solution of \eqref{nueq}. If $\psi_{\epsilon}^T=0$, substituting $f_{\epsilon}=0$ in \eqref{formula}, we obtain
	\begin{align}\label{1}
		\int_{\Omega}z_{\epsilon}(T)\psi^T\,dx-\int_{\Omega}z^0\psi(0)\,dx=-\sum_{i=1}^2\iint_{\O_{i,d}\times (0,T)}z_{i,d}\gamma^i\,dxdt,\quad \forall\psi^T\in L^2(\Omega).
	\end{align} 
	Since $0$ is the minimizer of $F_{\epsilon}$ and $F_{\epsilon}(0)=0$, we see
	\begin{align}\label{33}
		\lim_{h \to 0^+}\frac{F_{\epsilon}(h\psi^T)}{h}=\epsilon\|\psi^T\|+\left(z^0,\psi(0)\right)-\sum_{i=1}^{2}\iint_{\O_{i,d}\times (0,T)}z_{i,d}\gamma^i\,dxdt\ge 0
	\end{align}
	for any $\psi^T\in L^2(\Omega)$. Then from \eqref{1} and \eqref{33}, we have
	\begin{align}\label{apctrl}
		\|z_{\epsilon}(T)\|\le\epsilon.
	\end{align}
	If $\psi_{\epsilon}^T\neq 0$, we have
	\begin{align*}
		DF_{\epsilon}(\psi_{\epsilon}^T)\cdot \psi^T=0
	\end{align*}
	for any $\psi^T\in L^2(\Omega)$, i.e.
	\begin{align}\label{min}
		\iint_{\O\times(0,T)}\psi_{\epsilon}\psi\,dxdt+\epsilon \left(\frac{\psi^T_{\epsilon}}{\|\psi^T_{\epsilon}\|},\psi^T\right)+\int_{\Omega}z^0\psi(0)\,dx-\sum_{i=1}^{2}\iint_{\O_{i,d}\times (0,T)}z_{i,d}\gamma^i\,dxdt=0.
	\end{align}
	Then we see from (\ref{formula}) and \eqref{min} that
	\begin{align}\label{35}
		\iint_{\O\times(0,T)}f_{\epsilon}\psi\,dxdt-\iint_{\O\times(0,T)}\psi_{\epsilon}\psi\,dxdt-\epsilon \left(\frac{\psi^T_{\epsilon}}{\|\psi^T_{\epsilon}\|},\psi^T\right)=\int_{\Omega}z_{\epsilon}(T)\psi^T\,dx
	\end{align}
for any $\psi^T\in L^2(\Omega).$ Substituting $f_{\epsilon}=\psi_{\epsilon}1_{\O\times(0,T)}$ in \eqref{35}, we deduce that $z_{\epsilon}(T)=-\epsilon\psi_{\epsilon}^T/\|\psi_{\epsilon}^T\|$ satisfies (\ref{apctrl}). 
	
	Furthermore, taking $\psi^T=\psi^T_{\epsilon}$ in \eqref{min}, we infer from Proposition \ref{pro} that
	\begin{align}
		\iint_{\O\times(0,T)}|f_{\epsilon}|^2\,dxdt\le C\left(\|z^0\|^2+\sum_{i=1}^{2}\iint_{\O_{i,d}\times (0,T)}\rho^2|z_{i,d}|^2\,dxdt\right),
	\end{align}
	which implies that the family of controls $\{f_{\epsilon}\}_{\epsilon>0}$ is bounded in $L^2(\O\times(0,T)).$ Thus, there exists a
	subsequence that weakly convergent to some $\hat f$ in $L^2(\O\times(0,T)).$ From (\ref{ineq}), there exists a function $\hat z\in X,$ such that
	\begin{equation*}
			z_{\epsilon_k}\to \hat{z} \,\,\text{weakly in} \  X;
			z_{\epsilon_k}(T)\rightharpoonup\hat{z}(T) \,\, \text{in}\  L^2(\Omega).
	\end{equation*}
	Then we conclude from $\|z_{\epsilon_k}(T)\|=\epsilon_k$ that
	$\hat{z}(T)= 0.$ Similarly, after passing to a subsequence if necessary, $\phi^i_{\epsilon_k}\rightharpoonup \hat{\phi}^i$ in $X$. It's easy to verify that $(\hat{z},\hat{\phi^1},\hat{\phi^2})$ is the solution corresponding to $\hat{f}$. Then the null controllability is proved.
	
	From Proposition \ref{pro} and passing to the subsequence if necessary, we can assume that
	\begin{equation*}
		\psi_{\epsilon_k}(0)\rightharpoonup \hat{\psi}_0, \quad  \text{in} \  L^2(\Omega)
	\end{equation*}
	and
	\begin{equation*}
		\begin{cases}
			\rho^{-1}\bar\gamma_{\epsilon_k}\rightharpoonup\rho^{-1}\hat{\bar\gamma},  &  \text{in}\  L^2(\O_{d}\times(0,T)),\ \text{if \eqref{con1} holds},\\
			\rho^{-1}\gamma^i_{\epsilon_k}\rightharpoonup\rho^{-1}\hat{\gamma}^i,  &  \text{in}\  L^2(\O_{i,d}\times(0,T)),\ \text{if \eqref{con2} holds}.
		\end{cases}
	\end{equation*}
	Then for any $f$ such that \eqref{exa} holds, it follows from \eqref{nu} that
	\begin{align}\label{58}
		&\iint_{\O\times(0,T)}f\hat{f}\,dxdt=\lim_{k\to \infty}\left( -\int_{\Omega}z^0\psi_{\epsilon_k}(0)\,dx+\sum_{i=1}^{2}\iint_{\O_{i,d}\times (0,T)}z_{i,d}\gamma_{\epsilon_k}^i\,dxdt\right).
	\end{align}
	Taking $f=\hat{f}$ in \eqref{58}, we can also obtain
	\begin{equation}\label{5}
		\iint_{\O\times(0,T)}|\hat f|^2\,dxdt=\lim_{k\to \infty}\left( -\int_{\Omega}z^0\psi_{\epsilon_k}(0)\,dx+\sum_{i=1}^{2}\iint_{\O_{i,d}\times (0,T)}z_{i,d}\gamma_{\epsilon_k}^i\,dxdt\right).
	\end{equation}
	Combining \eqref{58} and \eqref{5}, we see that $\hat f$ minimizes the $L^2(\O\times(0,T))$ norm in the family of the null controls for $z$. Moreover, since $J(f)$ is strictly convex, $\hat f$ is the unique solution to the extremal problem \eqref{op}\textendash\eqref{exa}.
\end{proof}

\begin{remark}
	The assumption (\ref{ass}) is natural. Indeed, we would like to get \eqref{exa} and simultaneously keep $y$ not too far
	from $y_{i,d}$ in $\O_{i,d}\times (0, T)$; consequently, it is reasonable to impose that the $y_{i,d}$ approach $\bar y$ in $\O_{i,d}$ as $t$ goes to $T$.
\end{remark}

	
\section{The semilinear case}
	The main aim of this section is to establish  the exact controllability to trajectory of problem \eqref{eq1} in the semilinear case.
	
	In the linear case, the cost functionals are convex and continuously differentiable such that \eqref{Nash1} is equivalent to \eqref{queq}. However, for the semilinear case,  the convexity of the functionals $J_i$ is not ensured. Thus, we need to introduce the definition of Nash quasi-equilibrium.
	\begin{de}
		For any given $f\in L^2(\O\times(0,T))$, a pair $(\bar v^1,\bar v^2)$ is a Nash quasi-equilibrium for the functionals $J_i$ associated with $f,$ if the condition \eqref{queq} is satisfied.
	\end{de}

	\subsection{The optimality system in the semilinear case}\label{sec}
	In this subsection, we will deduce an optimality system that describes the Nash quasi-equilibrium.
	
Let $H_i$ and $H$ be defined as in \eqref{H},  for any given $f\in L^2(\O\times (0,T)),$  if $(\bar v^1,\bar v^2)\in H$ is the Nash quasi-equilibrium associated to $f$,  then we have
	\begin{align}
		\iint_{\O_{i,d}\times (0,T)}(y(f;\bar v^1,\bar v^2)-y_{i,d})w^i\,dxdt+\alpha_{i}\iint_{\O_i
			\times (0,T)}\bar v^iv^i\,dxdt=0,\quad \forall v^i\in H_i,\label{Nash31}
	\end{align}
	where $w^i$ is the solution to the system
	\begin{align}\label{eq31}
		\begin{cases}
			w^i_t-\Delta w^i-\frac{\mu}{|x|^2}w^i=F^{\prime}(y)w^i+v^i1_{\O_i},&(x,t)\in Q,\\
			w^i=0,&(x,t)\in \Sigma,\\
			w^i(x,0)=0,&x\in \Omega.
		\end{cases}
	\end{align}
	In order to further simplified equality \eqref{Nash31}, we introduce the adjoint system of problem \eqref{eq31}:
	\begin{align}\label{ad31}
		\begin{cases}
			-\phi^i_t-\Delta \phi^i-\frac{\mu}{|x|^2}\phi^i=F^{\prime}(y)\phi^i+(y-y_{i,d})1_{\O_{i,d}},&(x,t)\in Q,\\
			\phi^i=0,&(x,t)\in \Sigma,\\
			\phi^i(x,T)=0,&x\in\Omega.
		\end{cases}
	\end{align}
By combining \eqref{eq31} and \eqref{ad31}, we can reformulate equality \eqref{Nash31} as follows
	\begin{align*}
		\iint_{\O_i\times (0,T)}(\phi^i+\alpha_i\bar v^i)v^i\,dxdt=0,\quad \forall v^i\in H_i,
	\end{align*}
	which implies that
	\begin{align}\label{125}
		\bar v^i=-\frac{1}{\alpha_i}\phi^i\big|_{\O_i\times(0,T)}.
	\end{align}
	Consequently, we obtain the following optimality system:
	\begin{align}\label{op3}
		\begin{cases}
			y_t-\Delta y-\frac{\mu}{|x|^2}y=F(y)+f1_{\O}-\sum_{i=1}^{2}\frac{1}{\alpha_i}\phi^i1_{\O_i},&(x,t)\in Q,\\
			-\phi^i_t-\Delta \phi^i-\frac{\mu}{|x|^2}\phi^i=F^{\prime}(y)\phi^i+(y-y_{i,d})1_{\O_{i,d}},&(x,t)\in Q,\\
			y=0,\phi^i=0,&(x,t)\in \Sigma,\\
			y(x,0)=y^0(x),\phi^i(x,T)=0,&x\in \Omega.
		\end{cases}
	\end{align}
	
	In what follows, we will prove the existence and uniqueness of solutions to problem \eqref{op3} under some suitable assumptions.
	\begin{pro}\label{th}
		Assume that $f\in L^2(\O\times(0,T))$ and $F\in W^{1,\infty}(\R)\cap C^1(\R)$. If the constants $\alpha_i$ are sufficiently large, then for any $y^0\in L^2(\Omega),$ problem $\eqref{op3}$ admits a weak solution $(y,\phi^1,\phi^2)\in X\times X\times X$, where $X$ is defined in \eqref{X}. Moreover, if $F\in W^{2,\infty}(\R)$,  then weak solution of problem \eqref{op3} is also unique.
	\end{pro}
	\begin{proof}
		For any given $u\in L^2(Q)$, we consider the following problem
		\begin{align}\label{4.6}
			\begin{cases}
				y_t-\Delta y-\frac{\mu}{|x|^2}y=F(u)+f1_{\O}-\sum_{i=1}^{2}\frac{1}{\alpha_i}\phi^i1_{\O_i},&(x,t)\in Q,\\
				-\phi^i_t-\Delta\phi^i-\frac{\mu}{|x|^2}\phi^i=F^{\prime}(u)\phi^i+(y-y_{i,d})1_{\O_{i,d}},&(x,t)\in Q,\\
				y=0,\phi^i=0,&(x,t) \in \Sigma,\\
				y(x,0)=y^0(x),\phi^i(x,T)=0,&x\in\Omega.
			\end{cases}
		\end{align}
By carrying out the similar proof of the well-posedness of problem \eqref{nueq}, we conclude that there exists a unique weak solution $(y_u,\phi_u^1,\phi_u^2)\in X\times X\times X$. Define $\Lambda:L^2(Q)\to L^2(Q)$ by $\Lambda u:=y_u,$ then the mapping $\Lambda$ is well-defined, since the solution to \eqref{4.6} is unique. By the energy methods, we obtain
		\begin{align}\label{4.7}
			\|y_u\|_X\le C_0\left(\|f\|_{L^2(\O)}+1\right),
		\end{align} 
		where $C_0$ is a positive constant independent of $u$. 
		
Let $\mathcal{K}>0$ be a positive constant with $\mathcal{K}=\inf\{\lambda>0:\|u\|_{L^2(Q)}\leq \lambda\|u\|_X,\,\,\,\,\forall\,\,\,u\in X\},$ denote by $$R:=C_0\mathcal{K}(\|f\|_{L^2(\O)}+1)$$ and write $$\mathcal{B}:=\{u\in L^2(Q):\|u\|_{L^2(Q)}\le R\},$$ then $\Lambda$ maps $\mathcal{B}$ into itself.
  
  In the sequel, we will prove the existence of the solution to problem \eqref{4.6} by Leray-Schauder's fixed point theorem. Since 
  $$\mathcal{M}\hookrightarrow\hookrightarrow L^2(\Omega) \hookrightarrow\hookrightarrow \mathcal{M}^{\prime},$$  we have $X\hookrightarrow\hookrightarrow L^2(Q)$ from the Aubin-Lions compactness lemma.  Let $B$ be a bounded subset of $L^2(Q),$ we conclude from inequality \eqref{4.7} that $\|\Lambda u\|_X\le \frac{R}{\mathcal{K}}$ for all $u\in B$, which implies that $\Lambda(B)$ is relative compact in $L^2(Q)$.

Assume that $\{u_k\}_{k=1}^{\infty}$ convergent to $u$ in $L^2(Q),$ denote by $(y_k,\phi^1_k,\phi^2_k)$ the solution to
	 problem
		\begin{align}
			\begin{cases}
				y_{k,t}-\Delta y_k-\frac{\mu}{|x|^2}y_k=F(u_k)+f1_{\O}-\sum_{i=1}^{2}\frac{1}{\alpha_i}\phi^i_k1_{\O_i},&(x,t)\in Q,\\
				-\phi^i_{k,t}-\Delta\phi^i_k-\frac{\mu}{|x|^2}\phi^i_k=F^{\prime}(u_k)\phi_k^i+(y_k-y_{i,d})1_{\O_{i,d}},&(x,t)\in Q,\\
				y_k=0,\phi_k^i=0,&(x,t) \in \Sigma,\\
				y_k(x,0)=y^0(x),\phi^i_k(x,T)=0,&x\in\Omega,
			\end{cases}
		\end{align}
		then we conclude that $\{y_k\}_{k=1}^{\infty}$ is bounded in $X$ and relative compact in $L^2(Q)$. Then there exists a subsequence (still denoted by $\{y_{k}\}_{k=1}^{\infty}$),  such that
		\begin{equation*}
			\begin{cases}
				y_{k}\to \widetilde{y}&\text{in}\  L^2(Q),\\
				y_{k}\to \widetilde{y}&\text{weakly in}\ X.
			\end{cases}
		\end{equation*}
		Likewise, passing to a subsequence if necessary, we have
		\begin{equation*}
			\begin{cases}
				\phi^i_{k}\to \widetilde{\phi^i}&\text{in}\  L^2(Q),\\
				\phi^i_{k}\to \widetilde{\phi^i}&\text{weakly in}\ X.
			\end{cases}
		\end{equation*}
		Since $F\in W^{1,\infty}(\R)\cap C^1(\R)$, we conclude that $F:L^2(Q)\to L^2(Q)$ and $F^{\prime}:L^2(Q)\to L^{\frac{n}{2}}(Q)$ are Nemytski operators, then
		\begin{align*}
			\begin{cases}
				F(u_k) \to F(u), &  \text{in }L^{2}(Q),\\
				F^{\prime}(u_k) \to F^{\prime}(u), &  \text{in }L^{\frac{n}{2}}(Q).
			\end{cases}
		\end{align*}
		Then we can verify that $\widetilde{y}$ is the weak solution of \eqref{4.6} associated with $u$, that is, $\tilde{y}=y_u$. Therefore, $y_n\to y_u$ in $L^2(Q)$, which implies that $\Lambda:\mathcal{B}\to\mathcal{B}$ is continuous. 
		
		 Observe that $\Lambda$ satisfies the assumptions of Leray-Schauder's fixed point theorem and, consequently,  it possesses at least one fixed point $\bar y$. Furthermore, if $F\in W^{2,\infty}(\R)$, it's easy to obtain the uniqueness of solution by energy methods. As a consequence, problem \eqref{op3} admits a weak solution and the solution is unique if $F\in W^{2,\infty}(\R).$
	\end{proof}
	\begin{remark}
		From the Proposition \ref{th} and arguments in this section, we obtain the existence of Nash quasi-equilibrium. Furthermore, if $F\in W^{2,\infty}(\R)$, we could also obtain the uniqueness of Nash quasi-equilibrium.
	\end{remark}
	
	\subsection{Equilibrium and quasi-equilibrium}
	The main objective of this subsection is to prove that the concepts of Nash equilibrium and Nash quasi-equilibrium are equivalent if Nash quasi-equilibrium is unique.
	\begin{pro}
Assume that $F\in W^{2,\infty}(\R)$ and $\alpha_i$ are large enough, then for any given $f\in L^2(\O\times(0,T))$ and $y^0\in L^2(\Omega)$,  the couple $(\bar v^1,\bar v^2)$ is a Nash equilibrium if and only if it satisfies \eqref{queq}.
	\end{pro} 
	\begin{proof}
		It's obvious that Nash equilibrium satisfies \eqref{queq}, then we just need to prove the Nash quasi-equilibrium $(\bar v^1,\bar v^2)$ satisfies \eqref{Nash1}.
		
First of all, we will verify that the functional $\bar J_1:H_1\to \R$ given by $$\bar J_1(v^1)=\frac{1}{2}\iint_{\O_{1,d}\times(0,T)}|y(f;v^1,\bar{v}^2)-y_{1,d}|^2\,dxdt+\frac{\alpha_1}{2}\iint_{\O_1\times(0,T)}|v^1|^2\,dxdt$$ is weakly lower semi-continuous. For any $v^1_k\rightharpoonup v^1$ in $H_1$, we denote by $y_k$ the solution to system 
		\begin{align}\label{y}
			\begin{cases}
				y_{k,t}-\Delta y_k-\frac{\mu}{|x|^2}y_{k}=F(y_k)+f1_{\O}+v_k^11_{\O_1}+\bar{v}^21_{\O_2},&(x,t)\in Q,\\
				y_k=0,&(x,t) \in\Sigma,\\
				y_k(x,0)=y^0(x), &x\in\Omega,
			\end{cases}
		\end{align}
		it follows from the energy method that
		\begin{align*}
			\|y_k\|_X\le C(\|f\|_{L^2(\O\times(0,T))}+\|\bar{v}^2\|_{H_2}+\|v^1_k\|_{H_1}+1).
		\end{align*}
		Since $\{v^1_k\}_{k=1}^{\infty}$ is bounded in $H_1$, we can deduce that there exists a subsequence $\{y_k\}_{k=1}^{\infty}$ (still denote by themselves) such that 
		\begin{equation*}
			\begin{cases}
				y_{k}\to y&\text{in}\  L^2(Q),\\
				y_{k}\to y&\text{weakly in}\ X.
			\end{cases}
		\end{equation*}
Since $F\in W^{2,\infty}(\R)$ implies that $F:L^2(Q)\to L^2(Q)$ is a Nemytski operator, then
		\begin{align*}
			F(y_k) \to F(y)\quad  \text{in }L^{2}(Q).
		\end{align*}
		{}
	    Hence $y$ is the solution to problem
	    \begin{align*}
	    	\begin{cases}
	    		y_{t}-\Delta y-\frac{\mu}{|x|^2}y=F(y)+f1_{\O}+v^11_{\O_1}+\bar{v}^21_{\O_2},&(x,t)\in Q,\\
	    		y=0,&(x,t) \in\Sigma,\\
	    		y(x,0)=y^0(x), &x\in\Omega.
	    	\end{cases}
	    \end{align*}
	    Then we obtain
		\begin{align*}
			\varliminf_{k\to \infty} \bar J_1(v^1_k)&=\frac{1}{2}\varliminf_{k\to \infty}\iint_{\O_{1,d}\times(0,T)}|y_k-y_{1,d}|^2\,dxdt+\frac{\alpha_1}{2}\varliminf_{k\to \infty}\iint_{\O_1\times(0,T)}|v_k^1|^2\,dxdt\\
			&\ge \frac{1}{2}\iint_{\O_{1,d}\times(0,T)}|y-y_{1,d}|^2\,dxdt+\frac{\alpha_1}{2}\iint_{\O_1\times(0,T)}|v^1|^2\,dxdt\\
			&=\bar J_1(v^1).
		\end{align*}
		Observe that $\bar J_1$ is also coercive on $H_1$. Then $\bar J_1:H_1\to \R$ possesses a minimizer $\tilde{v}^1$. Since $\bar J_1$ is differentiable, we have 
		\begin{align*}
			D\bar J_1(\tilde{v}^1)\cdot v^1=0,\quad \forall v^1\in H_1.
		\end{align*}
		{}
		Arguing as in Section \ref{sec}, we can deduce that
		\begin{align}
			\tilde v^1=-\frac{1}{\alpha_1}\phi^1\big|_{\O_1\times(0,T)},
		\end{align}
		where $\phi^1$ is the solution to
		\begin{align}
			\begin{cases}
				-\phi^1_t-\Delta \phi^1-\frac{\mu}{|x|^2}\phi^1=F^{\prime}(y(f;\tilde{v}^1,\bar v^2))\phi^1+(y(f;\tilde{v}^1,\bar v^2)-y_{1,d})1_{\O_{1,d}},&(x,t)\in Q,\\
				\phi^1=0,&(x,t)\in \Sigma,\\
				\phi^1(x,T)=0,&x\in\Omega.
			\end{cases}
		\end{align}
		Consequently, $(y(f;\tilde v^1,\bar v^2),\phi^1)$ is the solution to the following system:
		\begin{align}\label{123}
			\begin{cases}
				y_t-\Delta y-\frac{\mu}{|x|^2}y=F(y)+f1_{\O}-\frac{1}{\alpha_1}\phi^11_{\O_1}+\bar v^21_{\O_2},&(x,t)\in Q,\\
				-\phi^1_t-\Delta \phi^1+\frac{\mu}{|x|^2}\phi^1=F^{\prime}(y)\phi^1+(y-y_{1,d})1_{\O_{1,d}},&(x,t)\in Q,\\
				y=0,\phi^1=0,&(x,t)\in \Sigma,\\
				y(x,0)=y^0(x),\phi^1(x,T)=0,&x\in \Omega.
			\end{cases}
		\end{align}
By carrying out the similar proof of Proposition \ref{th}, we can obtain the existence and uniqueness of the solution to problem
		\eqref{123}. Notice that $$\bar v^1=-\frac{1}{\alpha_1}\phi^1\big|_{\O_1\times(0,T)},$$ it infers from the uniqueness of the solution to \eqref{123} that $\bar v^1=\tilde{v}^1$.

		Similarly, taking 
		\[\bar J_2(v^2)=\frac{1}{2}\iint_{\O_{2,d}\times(0,T)}|y(f;\bar v^1,v^2)-y_{2,d}|^2\,dxdt+\frac{\alpha_2}{2}\iint_{\O_2\times(0,T)}|v^2|^2\,dxdt,\]
		then $\bar J_2:H_2\to \R$ possesses a minimizer $\widetilde {v}^2=\bar{v}^2$. Hence the pair $(\bar v^1,\bar v^2)$ fulfills \eqref{Nash1}, that is, $(\bar v^1,\bar v^2)$ is the Nash equilibrium.
	\end{proof}

	\subsection{Exact controllability to trajectory}
	We will prove the exact controllability to trajectory of problem \eqref{eq1} in this subsection.
	\begin{thm}\label{th2}
		Suppose that $\O_{i,d}$ and $\alpha_i$ are the same as in Theorem \ref{th1}, $F\in C^1(\R)\cap W^{1,\infty}(\R)$ and let $\bar y$ be the unique solution of problem \eqref{tr} with initial data $\bar y_0\in L^2(\Omega)$. If \eqref{ass} holds, then for any $y^0\in L^2(\Omega)$, there exists a control $f\in L^2(\O\times(0,T))$ and an associated Nash quasi-equilibrium $(\bar v^1,\bar v^2)$ such that the solution to \eqref{eq1} satisfies \eqref{exa}. Moreover, if $F\in W^{2,\infty}(\R)$, $(\bar v^1,\bar v^2)$ is also the Nash equilibrium.
	\end{thm}
	\begin{proof}
		Let us perform the change of variables $z=y-\bar y$ , we see that $(z,\phi^1,\phi^2)$ is the solution to problem
		\begin{align}
			\begin{cases}
				z_t-\Delta z-\frac{\mu}{|x|^2}z=G(x,t;z)z+f1_{\O}-\sum_{i=1}^{2}\frac{1}{\alpha_i}\phi^i1_{\O_i},&(x,t)\in Q,\\
				-\phi^i_t-\Delta\phi^i-\frac{\mu}{|x|^2}\phi^i=F^{\prime}(\bar y+z)\phi^i+(z-z_{i,d})1_{\O_{i,d}},&(x,t)\in Q,\\
				z=0,\phi^i=0,&(x,t) \in \Sigma,\\
				z(x,0)=z^0(x),\phi^i(x,T)=0,&x\in\Omega,
			\end{cases}\label{op31}
		\end{align}
		where $z=y-\bar y^0,\ z_{i,d}=y_{i,d}-\bar y$ and 
		\begin{align*}
			G(x,t;z)=\int_{0}^{1}F^{\prime}(\bar y+\tau z)\,d\tau.
		\end{align*}
		For each $z\in L^2(Q)$ and $f\in L^2(\O)$, we consider the linear system
		\begin{align}
			\begin{cases}
				w_t-\Delta w-\frac{\mu}{|x|^2}w=G(x,t;z)w+f1_{\O}-\sum_{i=1}^{2}\frac{1}{\alpha_i}\phi^i1_{\O_i},&(x,t)\in Q,\\
				-\phi^i_t-\Delta\phi^i-\frac{\mu}{|x|^2}\phi^i=F^{\prime}(\bar y+z)\phi^i+(w-z_{i,d})1_{\O_{i,d}},&(x,t)\in Q,\\
				w=0,\phi^i=0,&(x,t) \in \Sigma,\\
				w(x,0)=z^0(x),\phi^i(x,T)=0,&x\in\Omega.
			\end{cases}\label{eq33}
		\end{align}
Denote by $(w_z,\phi^1_z,\phi^2_z)$ the solution to system (\ref{eq33}),
		then we have
		\begin{align}\label{w}
			\|w_z\|_X\le C\left(\|f\|_{L^2(\O\times (0,T))}+1\right),
		\end{align}
		where $C$ is a positive constant independent of $z$. Let $(\psi_z,\gamma^1_z,\gamma^2_z)$ is the solution to problem
		\begin{align}\label{44}
			\begin{cases}
				-\psi_{z,t}-\Delta \psi_z-\frac{\mu}{|x|^2}\psi_z=G(x,t;z)\psi_z+\sum_{i=1}^{2}\gamma^i_z1_{\O_{i,d}},&(x,t)\in Q,\\
				\gamma^i_{z,t}-\Delta \gamma^i_z-\frac{\mu}{|x|^2}\gamma_z^i=F^{\prime}(\bar y+z)\gamma_z^i-\frac{1}{\alpha_i}\psi_z1_{\O_i},&(x,t) \in Q,\\
				\psi_z=0,\gamma_z^i=0,&(x,t)\in\Sigma,\\
				\psi_z(x,T)=\psi^T(x),\gamma^i_z(x,0)=0,&x\in \Omega.
			\end{cases}	
		\end{align}
		Combining problem \eqref{eq33} and \eqref{44}, we obtain that for any $\psi^T\in L^2(\Omega)$,
		\begin{align*}
			\iint_{\O\times (0,T)}f\psi_z\,dxdt-\sum_{i=1}^{2}\int_{\O_{i,d}\times(0,T)}\gamma^i_zz_{i,d}\,dxdt=-\int_{\Omega}z^0\psi_z(0)\,dx+\int_{\Omega}w_z(T)\psi^T\,dx,
		\end{align*}
		which entails that we have the null controllability of the problem \eqref{eq33} if and only if
		\begin{align}
			\iint_{\O\times (0,T)}f\psi_z\,dxdt-\sum_{i=1}^{2}\int_{\O_{i,d}\times(0,T)}\gamma^i_zz_{i,d}\,dxdt=-\int_{\Omega}z^0\psi_z(0)\,dx,\quad \forall \psi^T\in L^2(\Omega).
		\end{align}
		As in the previous section, we can define the functional
		\begin{align*}
			F_{\epsilon,z}(\psi^T)=\frac{1}{2}\iint_{\O\times(0,T)}|\psi_z|^2\,dxdt+\epsilon\|\psi^T\|+\int_{\Omega}z^0\psi_z(0)\,dx-\sum_{i=1}^{2}\int_{\O_{i,d}\times(0,T)}\gamma^i_zz_{i,d}\,dxdt.
		\end{align*}
		By Lemma \ref{carle} and carrying out the similar proof of Proposition \ref{pro}, we obtain the observability inequalities:
		\begin{align*}
				\|\psi_z(0)\|^2+\sum_{i=1}^{2}\iint_{\O_{i,d}\times(0,T)}\rho^{-2}|\gamma_z^i|^2\,dxdt\le C\iint_{\O\times(0,T)}|\psi_z|^2\,dxdt,  &  \quad \text{if \eqref{con1} holds,} \\
				\|\psi_z(0)\|^2+\iint_{\O_{i,d}\times(0,T)}\rho^{-2}|\sum_{i=1}^{2}\gamma_z^i|^2\,dxdt\le C\iint_{\O\times(0,T)}|\psi_z|^2\,dxdt,  &  \quad \text{if \eqref{con2} holds,}
		\end{align*}
		where $C$ is a positive constant independent of $z$. Arguing as in the proof of Theorem \ref{th1}, we get a leader $f_z\in L^2(\O\times(0,T))$ such that the associated solution to \eqref{eq33} satisfies
		\begin{align*}
			w_z(T)=0,\quad \text{for \textit{a.e.}}\  x\in \Omega.
		\end{align*}
		Moreover, there exists a positive constant $C$ independent of $z$, such that
		\begin{align}
			\|f_z\|_{L^2(\O\times (0,T))}\le C,\quad \forall z\in L^2(\O\times (0,T)).
		\end{align}
		Applying the Leray-Schauder's fixed point theorem, we can deduce that for any $z^0\in L^2(\Omega)$, there exists at least one control $f\in L^2(\O\times(0,T))$ such that the corresponding solutions to problem \eqref{op31} satisfies 
		\begin{align*}
			z(T)=0,\quad \text{for \textit{a.e.}}\  x\in \Omega.
		\end{align*}
		The details of proof are very similar with the proof of Theorem \ref{th}, we omit it here.
	\end{proof}
	\begin{remark}
		In fact, we can argue as in Theorem \ref{th1} that the control $f$ in Theorem \ref{th2} is the solution to the extremal problem \eqref{op}\textendash\eqref{exa}. However, the uniqueness is not obtained since we can not guarantee the convexity of $J(f)$.
	\end{remark}
\section*{Acknowledgement}
This work was supported by the National Science Foundation of China Grant (11871389), the Fundamental Research Funds for the Central Universities (xzy012022008) and Shaanxi Fundamental Science Research Project for Mathematics and Physics (22JSY032).

	\bibliographystyle{abbrv}
	\bibliography{BIB.bib}
\end{document}